\documentclass[reqno]{amsart}

\usepackage{tikz-cd}
\usepackage{fullpage,dsfont}
\usepackage{mathdots}
\usepackage{wasysym}

\usepackage{hyperref} % this should come before amsrefs (which loads

\usepackage[nobysame,alphabetic,initials,msc-links]{amsrefs}

\DefineSimpleKey{bib}{how}
\renewcommand{\PrintDOI}[1]{%
  \href{http://dx.doi.org/#1}{{\tt DOI:#1}}%
}
\renewcommand{\eprint}[1]{#1}
\BibSpec{book}{%
    +{}  {\PrintPrimary}                {transition}
    +{.} { \PrintDate}                  {date}
    +{.} { \textit}                     {title}
    +{.} { }                            {part}
    +{:} { \textit}                     {subtitle}
    +{,} { \PrintEdition}               {edition}
    +{}  { \PrintEditorsB}              {editor}
    +{,} { \PrintTranslatorsC}          {translator}
    +{,} { \PrintContributions}         {contribution}
    +{,} { }                            {series}
    +{,} { \voltext}                    {volume}
    +{,} { }                            {publisher}
    +{,} { }                            {organization}
    +{,} { }                            {address}
    +{,} { }                            {status}
    +{,} { \PrintDOI}                   {doi}
    +{,} { \PrintISBNs}                 {isbn}
    +{}  { \parenthesize}               {language}
    +{}  { \PrintTranslation}           {translation}
    +{;} { \PrintReprint}               {reprint}
    +{.} { }                            {note}
    +{.} {}                             {transition}
    +{}  {\SentenceSpace \PrintReviews} {review}
}
\BibSpec{article}{%
    +{}  {\PrintAuthors}                {author}
    +{,} { \textit}                     {title}
    +{.} { }                            {part}
    +{:} { \textit}                     {subtitle}
    +{,} { \PrintContributions}         {contribution}
    +{.} { \PrintPartials}              {partial}
    +{,} { }                            {journal}
    +{}  { \textbf}                     {volume}
    +{}  { \PrintDatePV}                {date}
    +{,} { \issuetext}                  {number}
    +{,} { \eprintpages}                {pages}
    +{,} { }                            {status}
    +{,} { \PrintDOI}                   {doi}
    +{,} { \eprint}        {eprint}
    +{}  { \parenthesize}               {language}
    +{}  { \PrintTranslation}           {translation}
    +{;} { \PrintReprint}               {reprint}
    +{.} { }                            {note}
    +{.} {}                             {transition}
    +{}  {\SentenceSpace \PrintReviews} {review}
}
\BibSpec{collection.article}{%
    +{}  {\PrintAuthors}                {author}
    +{,} { \textit}                     {title}
    +{.} { }                            {part}
    +{:} { \textit}                     {subtitle}
    +{,} { \PrintContributions}         {contribution}
    +{,} { \PrintConference}            {conference}
    +{}  {\PrintBook}                   {book}
    +{,} { }                            {booktitle}
    +{,} { \PrintDateB}                 {date}
    +{,} { pp.~}                        {pages}
    +{,} { }                            {publisher}
    +{,} { }                            {organization}
    +{,} { }                            {address}
    +{,} { }                            {status}
    +{,} { \PrintDOI}                   {doi}
    +{,} { \eprint}        {eprint}
    +{}  { \parenthesize}               {language}
    +{}  { \PrintTranslation}           {translation}
    +{;} { \PrintReprint}               {reprint}
    +{.} { }                            {note}
    +{.} {}                             {transition}
    +{}  {\SentenceSpace \PrintReviews} {review}
}
\BibSpec{misc}{%
  +{}{\PrintAuthors}  {author}
  +{,}{ \textit}      {title}
  +{.}{ }             {how}
  +{}{ \parenthesize} {date}
  +{,} { available at \eprint}        {eprint}
  +{,}{ available at \url}{url}
  +{,}{ }             {note}
  +{.}{}              {transition}
}

\usepackage{amssymb, amsfonts, amsxtra, amsmath}
\usepackage{mathrsfs}

\usepackage[all]{xy}

\usepackage{bbm}

\numberwithin{equation}{section}

\newtheorem{Theorem}{Theorem}[section]
\newtheorem{Def}[Theorem]{Definition}
\newtheorem{Lem}[Theorem]{Lemma}
\newtheorem{Prop}[Theorem]{Proposition}
\newtheorem{Cor}[Theorem]{Corollary}

\newtheorem{Rem}[Theorem]{Remark}

\newcommand\bp{\begin{proof}}
\newcommand\ep{\end{proof}}

\mathchardef\mhyph="2D

\usepackage{physics}

\usepackage[normalem]{ulem}

\date{}

\newcommand{\C}{\mathbb{C}}

\newcommand{\id}{\mathrm{id}}
\newcommand{\End}{\mathrm{End}}
\newcommand{\Mor}{\mathrm{Mor}}

\newcommand{\can}{\mathrm{can}}

\newcommand{\opp}{\mathrm{op}}
\newcommand{\splitting}{\mathrm{sp}}

\begin{document}
\title{A correspondence between homogeneous and Galois coactions of Hopf algebras}

\author{Kenny De Commer}
\address{Vrije Universiteit Brussel, Vakgroep wiskunde}
\email{kenny.de.commer@vub.be}

\author{Johan Konings}
\address{Vrije Universiteit Brussel, Vakgroep wiskunde}
\email{Johan.Konings@vub.be}

\thanks{The work of K.~De Commer was partially supported by the FWO grant G.0251.15N and the grant H2020-MSCA-RISE-2015-691246-QUANTUM DYNAMICS}

\begin{abstract}
A coaction of a Hopf algebra on a unital algebra  is called homogeneous if the algebra of coinvariants equals the ground field. A coaction of a Hopf algebra on a (not necessarily unital) algebra is called Galois, or principal, or free, if the canonical map, also known as the Galois map, is bijective. In this paper, we establish a duality between a particular class of homogeneous coactions, up to equivariant Morita equivalence, and Galois coactions, up to isomorphism. 
\end{abstract}

\maketitle

\section*{Introduction}

Galois coactions form a particular class of well-behaved coactions for Hopf algebras, providing at the same time a non-commutative generalization of freeness (also known as principality) of group actions on manifolds, and of the Galois condition for field extensions; see e.g. \cite{Schn90,Sch04} for an algebraic and \cite{BDCH17} for a topological approach. Their fundamental property is that they preserve tensor products upon considering associated vector bundles \cite{Ulb87,Ulb89,Sch04}. On the other hand, also the notion of homogeneity admits a direct generalization to the world of Hopf algebra coactions, by the requirement that the algebra of coinvariants is trivial. Especially in the setting of operator algebras, there has been considerable interest in homogeneous (also known as ergodic) actions of compact (quantum) groups on operator algebras, see e.g.~ \cite{Was89,Boc95,DCY13} and references therein.

In this paper, we want to set up a correspondence between homogeneous actions on the one hand, and free actions on the other. This correspondence is known in the operator algebraic setting, see already the fundamental work \cite{Was89} concerning compact group actions on von Neumann algebras, and see \cite[Section 9]{DC17} for a C$^*$-algebraic approach with respect to compact quantum group actions. However, in this paper we want to offer a purely algebraic approach, based upon the notion of \emph{equivariant semisimplicity}. In weeding out the analytic arguments, we hope that the fundamental nature of the above correspondence becomes clearer.

To keep the paper concise, we will not consider any concrete examples. However, we want to point out that the theory we construct is very closely related to the theory of \emph{weak multiplier Hopf algebras} \cite{VDW15,VDW17} and more precisely the theory of \emph{partial compact quantum groups} developed in \cite{DCT15}, which itself is a generalization of Hayashi's theory of compact quantum groups of face type \cite{Hay96}. We hope to revisit elsewhere the concrete correspondence and the associated construction of examples for the present theory, more specifically with respect to the notion of $I$-Galois object.

The paper is structured as follows. 

In the \emph{first section}, we recall some basic facts concerning algebras without unit, paying particular attention to the case of algebras with local units. Our main result concerns a Wedderburn decomposition for (absolutely) semisimple algebras which do not necessarily have a unit. 

In the \emph{second section}, we examine the notion of semisimplicity and Morita equivalence in the presence of a Hopf algebra coaction, called \emph{equivariant} or \emph{relative semisimplicity/Morita equivalence}. 

In the \emph{third section}, we establish our main results. We first introduce the notion of a homogeneous coaction of a Hopf algebra  on a unital algebra, being a coaction whose algebra of coinvariants equals the ground field. We then introduce the notion of an $I$-Galois object, being a Galois coaction whose algebra of coinvariants equals a direct sum of copies of the ground field. We then set up a one-to-one correspondence between equivariantly absolutely semisimple homogeneous coactions on the one hand, up to equivariant Morita equivalence, and (connected) $I$-Galois objects on the other hand, up to isomorphism. 

In the final \emph{fourth section}, we derive some structural results for $I$-Galois objects, notably the existence of invariant functionals and a Nakayama (or modular) automorphism.

\section{Semisimplicity for Algebras without Unit}

Fix a field $k$ and let $\otimes  = \underset{k}{\otimes}$ be the tensor product over $k$. Let $D$ be a $k$-algebra. We assume that $D$ is associative, but not necessarily unital. If $k\subseteq K$ is a field extension, we obtain by extension of scalars the $K$-algebra $D_K = K\underset{k}{\otimes} D$.

In the following definition we borrow the terminology of \cite{DVDZ99}. 

\begin{Def}
We call a $D$-module $V$ \emph{unital} if 
\[
V = DV := \{\sum_i x_i v_i\mid x_i\in D,v_i\in V\}.
\] 
We call a $D$-module $V$ \emph{non-degenerate} if $v\in V$ and $xv =  0$ for all $x\in D$ implies $v=0$. 

If $D$ is unital as a left module over itself by multiplication, we call $D$ \emph{idempotent}. We call $D$ \emph{non-degenerate} if it is non-degenerate as a left and as a right module over itself.
\end{Def}

\begin{Def}
We call a $D$-module \emph{simple} if $\{0\}\neq DV$ and $\{0\}$ and $V$ are the only submodules of $V$.

We call a $D$-module \emph{absolutely simple} if $V_K := K\otimes_k V$ is a simple $D_K$-module for all field extensions $k\subseteq K$. 

We call a $D$-module $V$ \emph{(absolutely) semisimple} if $V$ is a direct sum (possibly empty, possibly infinite) of (absolutely) simple $D$-modules. 

We call a $k$-algebra $D$ \emph{(absolutely) semisimple} if every unital $D$-module is (absolutely) semisimple.
\end{Def} 

Note that a semisimple module is automatically unital and non-degenerate. It is easy to see that a module $V$ is semisimple if and only if it is a (not necessarily direct) of simple submodules, see e.g.~ \cite[Theorem 4.2.3]{Cohn89}. We also have the following version of the Schur lemma and the Jacobson density theorem, see e.g. \cite[Remarks 4.2]{Dau72} and \cite[Theorem II.4]{Beh73}. We write $\Mor_D$ and $\End_D$ for the spaces of $k$-linear $D$-module maps.

\begin{Lem}\label{LemAbsEnk}
\begin{enumerate}
\item Let $V,W$ be simple $D$-modules. If $V,W$ are not isomorphic as $D$-modules, then $\Mor_D(V,W) = \{0\}$.
\item If $V$ is a simple $D$-module, then $\End_D(V)$ is a skew field.
\item If $V$ is a $D$-module, then the following assertions are equivalent. 
\begin{enumerate}
\item[(1)] $V$ is absolutely simple.
\item[(2)] $V$ is simple and $\End_D(V) = k\id_V$.
\item[(3)] For each $T\in \End_k(V)$ and each finite dimensional $k$-linear subspace $W \subseteq V$ there exists $x\in D$ with $xw = Tw$ for all $w\in W$.
\end{enumerate}
\end{enumerate}
\end{Lem} 

We will mainly be interested in a special class of non-unital algebras.

\begin{Def} 
Let $D$ be a $k$-algebra. We say $D$ has \emph{local units} if for each finite subset $\{x_1,\ldots,x_n\}\subseteq D$ there exists an \emph{idempotent} $e\in D$ with $e x_i = x_i = x_i e$ for all $i$.
\end{Def} 

The following lemma is straightforward.

\begin{Lem}\label{LemLocUn} 
Let $D$ be a $k$-algebra with local units. 
\begin{enumerate}
\item $D$ is idempotent and non-degenerate.
\item If $V$ is a unital $D$-module, then for each finite subset $\{v_i\}\subseteq V$ there exists an idempotent $e\in D$ with $ev_i = v_i$ for all $i$. In particular, $V$ is non-degenerate and every submodule of $V$ is also unital.
\end{enumerate}
\end{Lem} 
Let us recall the following terminology.

\begin{Def}  For $D$ a $k$-algebra, a $D$-module $V$ is called \emph{firm} if the map 
\[
D\otimes_D V\rightarrow V,\qquad d\otimes v \mapsto dv
\] 
is an isomorphism of $D$-modules.

We call $D$ a \emph{firm $k$-algebra} if it is firm as a left $D$-module over itself by multiplication.
\end{Def}
 
Clearly a firm module is unital. In the presence of local units, these notions are equivalent.

\begin{Lem}\label{LemFirm} 
Let $D$ be a $k$-algebra with local units. Then a $D$-module is firm if and only if it is unital.
\end{Lem} 

In particular, a $k$-algebra with local units is firm.

In the following we develop some Morita theory for $k$-algebras with local units. Morita equivalence for rings without unit has been treated at a number of places \cite{Abr83, Kyu86, AM87, GaSi91, GaVi98, Mar98}. We follow mainly \cite{AM87}, where the specific case of rings with local units is treated.

\begin{Def} 
We call \emph{Morita context} a quadruple $\{A_{ij}\}_{i,j\in \{1,2\}}$ of $k$-linear vector spaces $A_{ij}$ for $i,j\in \{1,2\}$, together with $k$-bilinear maps 
\[
A_{ij}\times A_{jk} \rightarrow A_{ik},\qquad (x,y) \mapsto xy
\] 
such that for all $i,j,k,l\in \{1,2\}$ the associativity condition 
\[
(xy)z = x(yz),\qquad x\in A_{ij},y\in A_{jk},z\in A_{kl}
\] 
is satisfied.
\end{Def}

Note that in particular $D = A_{11}$ and $E = A_{22}$ are $k$-algebras, while $M = A_{12}$ is a $D$-$E$-bimodule and $N = A_{21}$ is a $E$-$D$-bimodule. The extra data consist of a $D$-bimodule map 
\[
M \otimes_E N \rightarrow D,\quad x\otimes y \mapsto xy,
\]
and an $E$-bimodule map
\[
N \otimes_D M \rightarrow E,\quad x\otimes y \mapsto xy.
\] 
We then also say that $(M,N)$ together with the above maps is a Morita context between $D$ and $E$. If $k \subseteq K$ is a field extension, we have that $(M_K,N_K) := (K\otimes_kM,K\otimes_k N)$ is a Morita context between $D_K$ and $E_K$.

\begin{Def}
We call a Morita context $\{A_{ij}\}$ \emph{surjective} if $A_{ij}A_{jk} = A_{ik}$ for all $i,j,k\in \{1,2\}$. We call a Morita context \emph{strict} if the four multiplication maps 
\[
A_{ij} \otimes_{A_{jj}} A_{ji} \rightarrow A_{ii},\qquad i,j\in\{1,2\}
\]
are isomorphisms. 
\end{Def} 
If the Morita context $\{A_{ij}\}$ is surjective, this means in particular that the $A_{ii}$ are idempotent $k$-algebras, and that $A_{ij}$ is unital as a left $A_{ii}$-module and as a right $A_{jj}$-module. If the Morita context is strict, we have moreover that the $A_{ii}$ are firm $k$-algebras.

\begin{Def}
We call firm $k$-algebras $D,E$ \emph{Morita equivalent} if there exists a strict Morita context between $D$ and $E$.
\end{Def}  

If $(M,N)$ is a strict Morita context between $D$ and $E$, and if $k\subseteq K$ is a field extension, then $(M_K,N_K)$ is a strict Morita context between $D_K,E_K$.

It is easy to see that a strict Morita context $(M,N)$ between firm $k$-algebras $D$ and $E$ induces an equivalence between the associated categories of firm modules by 
\[
V \mapsto W = N\otimes_D V,\qquad W \mapsto V = M\otimes_E W,
\]
and that Morita equivalence is indeed an equivalence relation \cite{GaVi98}. It follows in particular from Lemma \ref{LemFirm} that if $D$ and $E$ are Morita equivalent $k$-algebras with local units, then $D$ is (absolutely) semisimple if and only if $E$ is (absolutely) semisimple. We also note the following theorem.

\begin{Theorem}[\cite{AM87}, Theorem 2.2]\label{TheoAM} 
Let $D,E$ be $k$-algebras with local units. Then a Morita context between $D$ and $E$ is strict if and only if it is surjective.
\end{Theorem}

We can easily characterize absolute semisimplicity through Morita equivalence. 

\begin{Theorem}\label{TheoSemisimp}
Let $D$ be a $k$-algebra with local units. Then $D$ is absolutely semisimple if and only if $D$ is Morita equivalent with $k_I$ for some set $I$, where $k_I = \oplus_{i\in I}k$ is the direct sum $k$-algebra over the (possibly infinite) index set $I$ of the trivial $k$-algebras $k$.
\end{Theorem}
Note that $k_I$ can be identified with the algebra of finite support $k$-valued functions on $I$. 
\begin{proof}
It is clear that the abelian category $\mathscr{D}$ of unital $D$-modules coincides with the category of unital $D$-modules for $D$ interpreted as a non-unital ring as in \cite{AM87}. On the other hand, the abelian category $\mathscr{K}_I$ of unital $k_I$-modules can be identified with the category of $I$-graded $k$-vector spaces. 

Let now $\{V_i\mid i\in I\}$ be a maximal collection of mutually non-isomorphic simple objects in $\mathscr{D}$. By absolute semisimplicity, and in particular Lemma \ref{LemAbsEnk}, it follows that
\[
F: \mathscr{D} \rightarrow \mathscr{K}_I,\quad  V \mapsto \oplus_{i\in I} \Mor_D(V_i,V),\qquad f \mapsto ((g_i)_i \mapsto (f\circ g_i)_i)
\]
defines an equivalence of categories. Hence $D$ and $k_I$ are Morita equivalent by \cite[Theorem 2.1]{AM87}. 
\end{proof}

We obtain from this theorem a Wedderburn decomposition for absolutely semisimple $k$-algebras with local units. We again first make an auxiliary definition. 

\begin{Def}\label{DefFinRank} 
Let $V$ be a vector space over $k$ with dual vector space $V^*$. Let $W\subseteq V^*$ be a subspace which separates the elements of $V$. Then we call an operator $T: V\rightarrow V$ of $W$-\emph{finite rank} if there are finitely many $\omega_i \in W$ and finitely many $v_i \in V$ such that 
\[
Tv = \sum_i \omega_i(v)v_i.
\] 
We denote by $M_W(V)$ the space of $W$-finite rank operators. 
\end{Def} 

For $\omega \in V^*$ and $v\in V$ we will in the following denote 
\[
T_{v,\omega}: V \rightarrow V,\qquad w\mapsto \omega(w)v.
\]
We then obtain that $M_W(V)$ is the linear span of all $T_{v,\omega}$ for $\omega\in W$ and $v\in V$. The space $M_W(V)$ is clearly an algebra, as  
\begin{equation}\label{EqCompFR} 
T_{v,\omega}T_{v',\omega'} = \omega(v') T_{v,\omega'}.
\end{equation} 

\begin{Lem} 
The algebra $M_W(V)$ has local units.
\end{Lem}
\begin{proof}If $Q \subseteq W$ and $P \subseteq V$ are finite-dimensional subspaces, there exist finite-dimensional subspaces $W \supseteq Q' \supseteq Q$ and $V \supseteq P' \supseteq P$ such that $Q' \cong (P')^*$ by restriction to $P'$. If then $(\omega_i,v_i)$ form dual bases of respectively $Q'$ and $P'$, then $e = \sum_i T_{v_i,\omega_i}$ is a local unit for the $T_{v,\omega}$ with $v\in P$ and $\omega \in Q$. 
\end{proof}

\begin{Theorem}\label{TheoMainSemiSimp} 
\begin{enumerate}
\item Let $\{V_i\mid i\in I\}$ be a collection of $k$-linear vector spaces, and let for each $i\in I$ be given a subspace $W_i \subseteq V_i^*$ separating the elements of $V_i$. Then the $k$-algebra 
\[
D = \oplus_{i\in I} M_{W_i}(V_i)
\] 
has local units and is absolutely semisimple. 
\item Each absolutely semisimple $k$-algebra with local units is isomorphic to an algebra of the above form.
\end{enumerate}
\end{Theorem} 

\begin{proof} 
(1) We have that  $D =  \oplus_{i\in I} M_{W_i}(V_i)$ has local units as each of its summands does. Putting $V  = \oplus_i V_i$ and $W = \oplus_i W_i$, we can easily define a surjective Morita context $(V,W)$ between $D$ and $k_I$. Hence $D$ and $k_I$ are Morita equivalent by Theorem \ref{TheoAM}, and in particular $D$ is absolutely semisimple. 

(2) Let $D$ be an absolutely semisimple $k$-algebra with local units. By Theorem \ref{TheoSemisimp}, there exists a strict Morita context $(V,W)$ between $D$ and $k_{I}$ for some index set $I$. The $k_I$-module structure gives decompositions $V = \oplus_i V_i$ and $W = \oplus_i W_i$. If now $\omega \in W_i$ and $\omega(v) = 0$ for all $v\in V_i$, it follows by the Morita equivalence that 
\[
0 = \omega \otimes v \in W\otimes_D V,\qquad \forall v\in V.
\]
Using that 
\[
W\otimes_D V \otimes_{k_I} W \cong W_{\otimes D} D \cong W,
\]
we conclude that $\omega = 0$. It follows that each pairing $W_i \times V_i \rightarrow k$ is non-degenerate, so that we can identify $W_i \subseteq V_i^*$. In a similar way it can be shown that $W_i$ separates the elements in $V_i$. As we have an isomorphism of algebras $V_i \otimes_k W_i \cong M_{W_i}(V_i)$, we obtain
\[
D \cong V \otimes_{k_I} W \cong \oplus_i (V_i \otimes W_i) \cong \oplus_i M_{W_i}(V_i). 
\]
\end{proof} 

\begin{Cor}\label{CorMorContSem} 
Let $D$ be an absolutely semisimple $k$-algebra with local units, and let $\{V_i\}_{i\in I}$ be a maximal collection of inequivalent simple $D$-modules. Let $(M,N)$ establish a Morita equivalence between $D$ and $k_{J}$. Then $I\cong J$ as sets, and $M \cong \oplus_{i\in I} V_i$ as a left $D$-module.
\end{Cor}

The Morita equivalence in Theorem \ref{TheoSemisimp} is obtained as a cutdown with a multiplier idempotent. More precisely, recall that the \emph{multiplier algebra} $M(D)$ of a $k$-algebra $D$ is the subalgebra 
\[
M(D) \subseteq \End_k(D)\oplus \End_k(D)^{\opp},
\] 
with $\End_k(D)$ the $k$-algebra of linear maps and $\End_k(D)^{\opp}$ the opposite $k$-algebra, consisting of $m = (\lambda,\rho)$ such that 
\[
\lambda(xy) = \lambda(x)y,\quad \rho(xy) = x\rho(y),\quad \lambda(\rho(x)) = \rho(\lambda(x)),\qquad \forall x,y\in D.
\] 
We then write 
\[
mx = \lambda(x),\qquad xm = \rho(x).
\]  
If $D$ is non-degenerate, and in particular if $D$ has local units, we obtain an embedding 
\[
D \rightarrow M(D),\quad x\mapsto x = (\lambda_x,\rho_x),\qquad \lambda_x(y) = xy,\quad \rho_x(y) = yx.
\] 

\begin{Def} 
If $D$ is a $k$-algebra with local units, we call an idempotent $p\in M(D)$ \emph{full} if $DpD = D$. 
\end{Def} 

It is easy to see that then $(Dp,pD)$ establishes (under the multiplication map) a  Morita equivalence between $D = DpD$ and $pDp$.  We have the following easy lemma. 

\begin{Lem}\label{LemFull}
If $D =\oplus_i M_{W_i}(V_i)$, then
\[
\prod_i M_{W_i}(V_i) \subseteq M(D),
\]
and with $\omega_i\in W_i,v_i \in V_i$ and $\omega_i(v_i) = 1$ for all $i$, the idempotent 
\[
p = \prod_i T_{v_i,\omega_i}  \in M(D)
\]
is a full idempotent establishing the Morita equivalence of $D$ and $pDp \cong k_I$.  
\end{Lem}

\section{Equivariant semisimplicity}

Let $(H,\Delta)$ be a Hopf $k$-algebra with counit $\varepsilon$ and antipode $S$, where we will use the Sweedler notation 
\[
\Delta(h) = h_{(1)}\otimes h_{(2)},\qquad h\in H.
\] 
We call (right) $H$-comodule algebra, or \emph{coaction}, any $k$-algebra $A$ endowed with a homomorphism 
\[
\alpha: A \rightarrow A\otimes H,\quad a \mapsto a_{(0)}\otimes a_{(1)}
\] 
that makes $A$ into an $H$-comodule.

\begin{Def} 
Let $(A,\alpha)$ be an $H$-comodule algebra. An \emph{$H$-equivariant, or $(H,A)$-relative module} for $(A,\alpha)$ consists of a $k$-linear vector space $V$ with an $A$-module structure and a right $(H,\Delta)$-comodule structure $\delta: V\rightarrow V\otimes H$ such that 
\[
\delta(av) = \alpha(a)\delta(v),\qquad a\in A,v\in V.
\]
\end{Def} 

We can extend the notions introduced for modules to $H$-equivariant $A$-modules: we call an $H$-equivariant $A$-module $(V,\delta)$ 

\begin{itemize}
\item unital if $V$ is unital as an $A$-module,
\item equivariantly simple if $\{0\}\neq AV$ and $\{0\}$ and $V$ are the only equivariant submodules of $V$,
\item equivariantly absolutely simple if $V_K$ is simple as an $H_K$-equivariant $A_K$-module for all field extensions $k\subseteq K$,
\item equivariantly semisimple if it is a direct sum of equivariantly simple $H$-equivariant $A$-modules,
\item equivariantly absolutely semisimple if it is a direct sum of equivariantly absolutely simple $H$-equivariant $A$-modules.
\end{itemize}

We call an $H$-comodule algebra $(A,\alpha)$ equivariantly (absolutely) semisimple if each unital $H$-equivariant $A$-module is equivariantly (absolutely) semisimple. 

The theory of the previous section could then be extended to the setting of $H$-equivariant $A$-modules. However, to apply the results \emph{directly}, we need to assume that $(H,\Delta)$ has more structure. We use the terminology of  \cite{VDZ99b}. 

\begin{Def} 
We say that a Hopf algebra $(H,\Delta)$ has \emph{invariant functionals}, or defines an \emph{algebraic quantum group of compact type}, if there exist non-zero functionals 
\[
\varphi: H \rightarrow k,\qquad \psi: H \rightarrow k,
\] 
called respectively \emph{left} and \emph{right} invariant functional, such that for all $h\in H$ 
\[
(\id\otimes \varphi)\Delta(h) = \varphi(h)1,\qquad (\psi\otimes \id)\Delta(h) = \psi(h)1.
\] 
We call $\varphi,\psi$ invariant \emph{integrals} if also $\varphi(1) = \psi(1) = 1$.
\end{Def} 

Note that the existence of $\varphi$ implies the existence of $\psi$, as one can take $\psi = \varphi \circ S$. We recall the following properties.

\begin{Prop}[\cite{VDZ99b}]\label{PropHopfInv} 
Let $(H,\Delta)$ be a Hopf algebra with invariant functionals $\varphi,\psi$. Then 
\begin{itemize}
\item the antipode $S$ is invertible,
\item the functionals $\varphi,\psi$ are unique up to multiplication with a scalar in $k^{\times}$,
\item there exists a (unique) invertible grouplike element $\delta\in H$ with 
\[
\varphi(S(h)) = \varphi(h\delta),\qquad \forall h\in H,
\]
which we call the \emph{modular element} or the \emph{distinguished grouplike},
\item there exists a (unique) algebra automorphism $\sigma$ of $H$ with 
\[
\varphi(ab) = \varphi(b\sigma(a)),\qquad \textrm{for all }a,b\in H,
\]
which we call the \emph{modular automorphism} or the \emph{Nakayama automorphism},
\item the functionals $\varphi,\psi$ are faithful, i.e.~ they span faithful left and right $H$-modules in the $k$-linear dual $H^*$ of $H$,
\item a left invariant integral $\varphi$ is also a right invariant integral, and in this case $\delta = 1$.
\end{itemize}
\end{Prop} 

In the following we fix a Hopf algebra $(H,\Delta)$ with invariant functionals $\varphi,\psi = \varphi\circ S$. The previous proposition implies the equality of the subsets 
\[ 
\{\varphi(-h)\mid h\in H\} = \{\varphi(h-)\mid h\in H\}  = \{\psi(-h)\mid h\in H\} = \{\psi(h-)\mid h\in H\}
\] 
inside the linear dual $H^*$  of $H$, and we write this set as $\hat{H}$. By the second to last item, the space of functionals $\hat{H}$ separates the elements of $H$. Moreover, if we consider $H^*$ with the convolution product 
\[
(\omega*\chi)(h) = (\omega\otimes \chi)\Delta(h),
\]
then $\hat{H}$ is a $k$-subalgebra of $H^*$ possessing local units (and non-unital for $H$ infinite-dimensional), see \cite[Proposition 3.1]{VDZ99a}. In fact, $\hat{H}$ is a 2-sided dense ideal in $H^*$, where the density means that $\hat{H}$ is faithful as a left or right $H^*$-module. 

The following lemma follows as in \cite[Proposition 3.3]{DVDZ99}, using that $\hat{H}$ is a two-sided ideal of $H^*$ with local units.

\begin{Lem}\label{LemExtHdual} 
Let $V$ be a unital left $\hat{H}$-module. Then there exists a unique extension to a unital left $H^*$-module.
\end{Lem} 

The algebra $\hat{H}$ actually has the structure of an \emph{algebraic quantum group of discrete type} \cite{VD98,VDZ99a}. For example, let 
\[
\hat{\Delta}: H^* \rightarrow (H\otimes H)^*,\quad \hat{\Delta}(\omega)(h,k) = \omega(hk).
\] 
Viewing $(H\otimes H)^*$ as the convolution dual of the tensor product Hopf algebra $H\otimes H$, we then have that 
\begin{equation}\label{EqMultHopf}
\hat{H}\otimes \hat{H} = \hat{\Delta}(\hat{H})(1\otimes \hat{H}) = \hat{\Delta}(\hat{H})(\hat{H}\otimes 1)   = (1\otimes \hat{H})\hat{\Delta}(\hat{H}) = (\hat{H}\otimes 1)\hat{\Delta}(\hat{H}).
\end{equation}

\begin{Lem}\label{LemDualModComod} 
Let $(H,\Delta)$ be a Hopf algebra with invariant functionals, and let $(V,\delta)$ be a right $H$-comodule. Then $V$ is a unital left $\hat{H}$-module under 
\[
\omega v  = (\id\otimes \omega)\delta(v).
\] 
Moreover, this provides a categorical equivalence between $H$-comodules and unital left $\hat{H}$-modules. 
\end{Lem} 

\begin{proof}
It is immediate that $V$ becomes a left $\hat{H}$-module. It is unital since $\hat{H}$ separates the elements of $H$, hence for any finite collection $h_1,\ldots,h_n$ there exists $\omega \in\hat{H}$ with $\omega(h_i) = \varepsilon(h_i)$ for all $i$. 

To see that any unital $\hat{H}$-module appears in this way, we follow the argument in \cite[Proposition 3.2]{VDZ99b}. Namely, note that $V\otimes \hat{H}$ is a unital $\hat{H}\otimes \hat{H}$-module, and can hence be extended to a $(H\otimes H)^*$-module (applying Lemma \ref{LemExtHdual} to the tensor product Hopf algebra $H\otimes H$). Since $V = \hat{H}V$ and $\hat{H}\otimes \hat{H} = \hat{\Delta}(\hat{H})(\hat{H}\otimes 1)$, there exists a unique linear map 
\[
\tilde{\delta}: V\rightarrow V\otimes \hat{H},\quad \widetilde{\delta}(v)(1\otimes \chi) = \hat{\Delta}(\varphi)(v\otimes \chi),\quad v\in V,\chi \in \hat{H}.
\] 
Let 
\[
\mathscr{F}: H \rightarrow \hat{H},\quad h \mapsto \varphi(h-).
\] 
Then $\mathscr{F}$ is a linear bijection with $\varphi(\mathscr{F}^{-1}(\omega) h) = \omega(h)$ for all $h\in H,\omega\in\hat{H}$. It follows that the map 
\[
\delta: V\rightarrow V\otimes H,\quad \delta(v) = (\id\otimes \mathscr{F}^{-1})\tilde{\delta}(v)
\] 
satisfies 
\[
(\id\otimes \omega)\delta(v) = \omega v,\qquad \omega \in\hat{H},v\in V.
\] 
As $V$ is a left $\hat{H}$-module and $\hat{H}$ separates elements in $H$, we must have that $(V,\delta)$ is a comodule. Then also the counit identity must be satisified since $\omega v = (\id\otimes \omega)\delta(v)$ is an extension to $H^*$ of the given $\hat{H}$-module, and must hence coincide with the unique, unital extension to $H^*$ of the $\hat{H}$-module structure on $V$.
\end{proof} 

A similar correspondence can be established for comodule algebras, using the theory of smash products for discrete algebraic quantum groups \cite{DVDZ99}. 

\begin{Def} 
Let $(H,\Delta)$ be a Hopf algebra with invariant functionals, and let $(A,\alpha)$ be an $H$-comodule algebra. The \emph{smash product} algebra $A\#\hat{H}$ is defined as the vector space $A\otimes \hat{H}$ endowed with the product 
\[
(a\otimes \omega)(b\otimes \chi) = ab_{(0)}\otimes \omega(b_{(1)}-)*\chi,\qquad a,b\in A,\omega,\chi \in \hat{H}.
\]
\end{Def}

From the definition of $\hat{H}$ it is clear that the above product is well-defined, and it is also easily checked that we get an associative product this way. When working with the smash product, we will use the notation $a\#\omega$ for elementary tensors. We will also use the notation
\[
\omega \#a := a_{(0)}\# \omega(a_{(1)}-), \qquad \omega \in \hat{H},a\in A.
\]
It is not hard to see that  $\hat{H}\#A  = A\#\hat{H}$, and that we get a linear isomorphism 
\[
\hat{H}\otimes A\rightarrow \hat{H}\# A,\quad \omega\otimes a \mapsto \omega\#a.
\]
The multiplication on $\hat{H}\#A$ is then determined by
\[
(\omega\#a)(\chi\#b) = (\omega*\chi)\left(S^{-1}(a_{(1)})-\right) \# a_{(0)}b.
\]
We will in the following use the notation $A\#\hat{H}$ and $\hat{H}\#A$ interchangeably. 

To be able to set up properly the correspondence between $H$-equivariant $A$-modules and modules over the smash product, we will assume the existence of a particular kind of local units. 

\begin{Def}\label{DefCoinv}
Let $(H,\Delta)$ be a Hopf algebra, and let $(A,\alpha)$ be an $H$-comodule algebra.  The \emph{coinvariant algebra} is defined as 
\[
A^{\alpha}= \{a\in A\mid \alpha(a) = a\otimes 1\}.
\]
We say that $A$ has \emph{$H$-coinvariant local units} if for each finite subset $\{a_i|i = 1,...,n\}\subset A$ there exists $e\in A^{\alpha}$ with $ea_i = a_i = a_ie$, for $i = 1,\ldots,n$. 
\end{Def} 

For example, this condition is satisfied when $A$ is a unital algebra.

\begin{Lem}\label{LemLocUnSmash}  
Let $(H,\Delta)$ be a Hopf algebra with invariant functionals, and let $(A,\alpha)$ be an $H$-comodule algebra with $H$-coinvariant local units. Then $A\#\hat{H}$ has local units.
\end{Lem} 
\begin{proof} 
 If $V \subseteq A$ and $X\subseteq \hat{H}$ are finite-dimensional subspaces, we can choose finite-dimensional subspaces  $W \subseteq A$ and $Y\subseteq \hat{H}$ such that $X\# V \subseteq W\#Y$. Choose $e\in A^{\alpha}$ an idempotent with $ea = a = ae$ for all $a\in V+W$, and choose $\chi \in \hat{H}$ an idempotent with $\chi* \omega = \omega  = \omega*\chi$ for all $\omega \in X+Y$. Then by $H$-coinvariance of $e$ it is clear that $e\#\chi = \chi\#e$ is a local unit with respect to $X\#V$.  
\end{proof} 

\begin{Rem}
The above lemma is implied in a remark following \cite[Lemma 5.2]{VDZ99b}, without the assumption that the local units in $A$ are $H$-coinvariant. However, in this case we fail to see how one can build \emph{idempotent} local units. 
\end{Rem}

We now have the following proposition by combining Lemma \ref{LemDualModComod} with \cite[Proposition 5.11]{DVDZ99}, together with the observation that a unital $A\#\hat{H}$-module is automatically $\hat{H}$-unital since, using notation as in Lemma \ref{LemLocUnSmash}, $A\#\hat{H} = \hat{H}\#A$ and $\hat{H}$ is idempotent.

\begin{Prop}\label{PropCatEqDi} 
Let $(H,\Delta)$ be a Hopf algebra with invariant functionals, and let $(A,\alpha)$ be an $H$-comodule algebra with $H$-coinvariant local units. If $V$ is a unital $H$-equivariant $A$-module, then $V$ becomes a unital $A\#\hat{H}$-module by 
\[
(a\#\omega)v = a(\omega v),
\] 
setting up a categorical equivalence between unital $H$-equivariant $A$-modules and unital $A\#\hat{H}$-modules. 
\end{Prop} 

Combining Proposition \ref{PropCatEqDi} with Theorem \ref{TheoSemisimp}, we obtain the following theorem.

\begin{Theorem}\label{TheoEqMor}  
Let $(H,\Delta)$ be a Hopf algebra with invariant functionals, and let $(A,\alpha)$ be an $H$-comodule algebra with $H$-coinvariant local units. Then the following are equivalent:
\begin{itemize}
\item $(A,\alpha)$ is $H$-equivariantly absolutely semisimple,
\item $A\#\hat{H}$ is absolutely semisimple as a $k$-algebra,
\item $A\#\hat{H}$ is Morita equivalent to $k_{I}$ for some index set $I$.
\end{itemize}
\end{Theorem} 

\begin{Rem}
It is well-known that if $\hat{H}$ is absolutely semisimple, necessarily $H$ must have an invariant integral, corresponding to the unique normalized element in the one-dimensional kernel of the counit of $\hat{H}$. However, this condition is not a necessary requirement for the existence of $H$-equivariant absolutely semisimple comodule algebras, as for example for the comodule algebra $H$ itself the associated smash product $H \#\hat{H} \cong \End_{\hat{H}}(H)$ is absolutely semisimple by \cite[Proposition 6.7]{DVDZ99}, where we use the notation from Definition \ref{DefFinRank}. 
\end{Rem}

Also Morita theory can be developed equivariantly. 

\begin{Def} 
Let $(H,\Delta)$ be a Hopf algebra. An \emph{$H$-equivariant, or $H$-relative Morita context} consists of a Morita context $\{A_{ij}\}$ where all $A_{ij}$ are $H$-comodules $(A_{ij},\alpha_{ij})$ and all multiplication maps $A_{ij}\otimes A_{jk}\rightarrow A_{ik}$ are $H$-comodule morphisms. It is called surjective, respectively strict if the underlying Morita context is surjective, respectively strict.
\end{Def} 

As before, we then have that $(A,\alpha) = (A_{11},\alpha_{11})$ and $(B,\beta) = (A_{22},\alpha_{22})$ are $H$-comodule algebras, and we call $((M,\delta_M),(N,\delta_N)) = ((A_{12},\alpha_{12}),(A_{21},\alpha_{21}))$ an $H$-equivariant Morita context between $(A,\alpha)$ and $(B,\beta)$. 

\begin{Def} 
We call a couple $(A,\alpha),(B,\beta)$ of $H$-comodule algebras \emph{$H$-equivariantly Morita equivalent} if there exists a strict $H$-equivariant Morita context between $(A,\alpha)$ and $(B,\beta)$. 
\end{Def}

It is again straightforward to show that $H$-equivariant Morita equivalence is indeed an equivalence relation, and that then the categories of equivariant firm modules are equivalent. This equivalence will however be \emph{stronger} in general than just Morita equivalence of the associated smash products, the point being that their categories of firm modules will be equivalent as $\mathrm{Comod}(H)$-module categories.

An example of an equivariant Morita equivalence is the following. Let $(A,\alpha)$ be an $H$-comodule algebra. Note that $A\#\hat{H}$ is naturally a (non-unital) left $H$-module algebra by 
\[
h \cdot (a\#\omega) = a\#\omega(-h). 
\]
One can hence form the iterated smash product $(A\#\hat{H})\#H$ where
\[
(x\#h)(y\#k) = x(h_{(1)}\cdot y)\# h_{(2)}k,\qquad x,y \in A\#H,h,k\in H.
\]  
This becomes an $H$-comodule algebra by 
\[
\widehat{\widehat{\alpha}}: (a\#\omega)\#h \mapsto (a\#\omega)\#h_{(1)}\otimes h_{(2)},
\]
with algebra of coinvariants 
\[
A\#\hat{H} \subseteq (A \#\hat{H})\#H.
\]
In particular, $(A\#\hat{H})\#H$ has $\widehat{\widehat{\alpha}}$-coinvariant local units if $A$ has $\alpha$-coinvariant units. Let us denote by $\widehat{\widehat{\alpha}}_{S^2}$ the $S^2$-twisted coaction
\[
\widehat{\widehat{\alpha}}_{S^2}: (a\#\omega)\#h \mapsto (a\#\omega)\#h_{(1)}\otimes S^2(h_{(2)}).
\]

\begin{Prop}\label{PropEqMorDual}
Let $(A,\alpha)$ be an $H$-comodule algebra with $H$-coinvariant local units. Then $(A,\alpha)$ and $((A\#\hat{H})\#H,\widehat{\widehat{\alpha}}_{S^2})$ are equivariantly Morita equivalent. 
\end{Prop} 
\begin{proof}
This is well-known in the operator algebraic setting as the \emph{biduality theorem} \cite{ES80}, where however one can avoid twisting with $S^2$. The latter is also possible in the algebraic setting whenever $\hat{S}^2$ is inner on $\hat{H}$. In general, however, it is not clear how to avoid the additional twisting. 

Let us sketch a proof of the above proposition using techniques from \cite[Section 4]{VDZ99b}. We refrain from spelling out all the computational details.

Consider for $A$ and $M = A\#\hat{H}$ as $k$-linear spaces the following operators:
\[
\lambda^A: A \times A \rightarrow A,\quad \lambda_A(a)b = ab,\qquad
\lambda^{\hat{H}}: \hat{H} \times A \rightarrow M,\quad \lambda_{\hat{H}}(\omega)a =  \omega \#a,
\]
\[
\lambda_{\#}^{A}: A \times M \rightarrow M,\quad \lambda_{\#}^A(a)b\#\omega = ab\#\omega,\qquad 
\lambda_{\#}^{\hat{H}}: \hat{H}\times M \rightarrow M,\quad \lambda_{\#}^{\hat{H}}(\omega)\chi \# a =  \omega*\chi \#a,
\]
\[\lambda_{\#}^{H}: H\times M \rightarrow M,\quad \lambda_{\#}^{H}(h)a\#\omega = a \# \omega(-h),  \qquad
\Lambda^{\hat{H}}: \hat{H}\times M \rightarrow A,\quad \Lambda^{\hat{H}}(\omega)\chi \#b = \hat{\varphi}(\omega*\chi)b,
\]
where $\hat{\varphi}(\psi(a-)) = \varepsilon(a)$ is the left invariant functional for the discrete algebraic quantum group $(\hat{H},\hat{\Delta})$. 

Let 
\[
D = (A\#\hat{H})\#H,\qquad N =\hat{H}\# A,\qquad V = \begin{pmatrix} M \\ A \end{pmatrix},\qquad Q = \begin{pmatrix} Q_{11} & Q_{12} \\ Q_{21} & Q_{22}\end{pmatrix} = \begin{pmatrix} D & M \\ N & A \end{pmatrix},
\]
where we consider the latter two entries simply as direct sums of vector spaces for the moment. Consider the linear map
\[
\pi: Q \rightarrow \End_k(V),\qquad \begin{pmatrix} (a \# \omega)\#h & b \#\chi \\ \theta \# c & d \end{pmatrix} \mapsto \begin{pmatrix} \lambda^{A}_{\#}(a)\lambda_{\#}^{\hat{H}}(\omega) \lambda_{\#}^H(h) & \lambda_{\#}^A(b) \lambda^{\hat{H}}(\chi) \\  \Lambda^{\hat{H}}(\theta)\lambda_\#^A(c) & \lambda^A(d) \end{pmatrix}. 
\]
It is not hard to verify that $\pi$ is an injective map. Moreover, by the concrete commutations relations between the above operators, we find that the image is a subalgebra of $\End_k(V)$. Endowing $Q$ with the inherited algebra structure, we obtain maps
\[
Q_{ij}\otimes Q_{jk} \rightarrow Q_{ik}.
\]
By basic computations it is seen that the resulting algebra structures on $A,D$ coincide with the usual ones. As it is also easily seen that this Morita context is surjective, it follows by Theorem \ref{TheoAM} that $A$ is Morita equivalent to $D$. We are hence left with showing that the Morita context $(M,N)$ can be made equivariant. 

As $D$ and $A$ have local units, and $M,N$ are unital bimodules, it follows that $Q$ has local units. By \cite[Theorem 3.3]{VDZ99b}, it is enough to produce a (unital) $\hat{H}$-module algebra structure on $Q$, coinciding with the $\hat{H}$-module structure on $D$ and $A$ induced by the $H$-coaction. However, consider on $V$ the unital left $\hat{H}$-module structure defined by 
\[
\Theta(\omega) \begin{pmatrix} a\#\chi \\ b\end{pmatrix} = \begin{pmatrix} a\# \chi*\hat{S}(\omega) \\ \omega \cdot b \end{pmatrix}, 
\]
 Using \eqref{EqMultHopf} and the bijectivity of the antipode of $\hat{H}$, we can turn $\End_k(V)$ into a well-defined unital left $\hat{H}$-module algebra by  the $\Theta$-adjoint action
\[
\omega \cdot x = \Theta(\omega_{(1)}) x\Theta(\hat{S}(\omega_{(2)})). 
\]
It is then easily seen that
\[
\omega \cdot \begin{pmatrix} 0 & 0 \\ 0 & \lambda^A(a) \end{pmatrix} = \begin{pmatrix} 0 & 0 \\ 0 & \lambda^A(\omega \cdot a) \end{pmatrix},\qquad \omega \cdot \begin{pmatrix} 0 & \lambda^{\hat{H}}(\chi) \\ 0 & 0 \end{pmatrix} = \begin{pmatrix} 0 & \lambda^{\hat{H}}(\chi *\hat{S}(\omega)) \\ 0 & 0 \end{pmatrix},
\]
\[
\omega \cdot \begin{pmatrix} 0 & 0 \\ \Lambda^{\hat{H}}(\chi)& 0 \end{pmatrix}  = \begin{pmatrix} 0 &  0 \\ \Lambda^{\hat{H}}(\hat{\sigma}^{-1}\hat{S}^2(\omega)*\chi) & 0 \end{pmatrix},
\]
where $\hat{\sigma}$ is the modular automorphism of the discrete quantum group $(\hat{H},\hat{\Delta})$. It follows that the above $\hat{H}$-module algebra structure restricts to a unital $\hat{H}$-module algebra structure on $Q$. Finally, we need to see if the obtained module structure on $D$ coincides with the one coming from the $H$-coaction on $D$. This is equivalent with the following commutation relations for $a\in A,\omega,\chi \in \hat{H},h\in H$, which are straightforward to check:
\[
\Theta(\omega)\lambda_{\#}^A(a) = \lambda_{\#}^{A}(a) \Theta(\omega),\quad \Theta(\omega)\lambda_{\#}^{\hat{H}}(\chi) = \lambda_{\#}^{\hat{H}}(\chi) \Theta(\omega),\quad  \Theta(\omega)\lambda_{\#}^{H}(h) = \lambda_{\#}^{H}(h_{(1)}) \Theta(\omega(S^2(h_{(2)})-)).
\]
\end{proof}

\section{$I$-Galois objects and homogeneous coactions}

In this section we will establish our main result, which sets up a duality between \emph{free} actions on the one hand, and \emph{homogeneous actions} on the other hand, within the setting of module coalgebras for Hopf algebras with invariant functionals. Recall the notation $A^{\alpha}$ from Definition \ref{DefCoinv}.

\begin{Def} 
Let $(H,\Delta)$ be a Hopf algebra, and $(A,\alpha)$ an $H$-comodule algebra.

We call $(A,\alpha)$ a \emph{Galois coaction}, or \emph{free coaction}, or \emph{principal coaction}, if the \emph{Galois map} 
\[
\can: A \otimes_{A^{\alpha}} A \rightarrow A \otimes H,\quad a\otimes b \mapsto (a\otimes 1)\alpha(b)
\] 
is an isomorphism.

We call $(A,\alpha)$ a \emph{homogeneous coaction}, or \emph{transitive coaction}, if $A$ is unital and $A^{\alpha} = k1_A$. 
\end{Def} 

One can indeed interpret the above notions, within the proper setting, as corresponding to freeness and transitivity of group actions, see \cite{Pod95,Sch04}. 

\begin{Rem} 
Since 
\[
xy_{(0)}\otimes y_{(1)} = (x_{(0)}y)_{(0)}\otimes S^{-1}(x_{(1)})(x_{(0)}y)_{(1)}
\] 
and the map 
\[
A\otimes H \rightarrow A\otimes H,\quad x\otimes h \mapsto x_{(0)}\otimes S^{-1}(h)x_{(1)}
\] 
is bijective, also the map 
\begin{equation}\label{EqGalOth}
\widetilde{\can}: A\underset{A^{\alpha}}{\otimes} A\rightarrow A\otimes H,\quad x\otimes y \rightarrow \alpha(x)(y\otimes 1)
\end{equation}
will be bijective if $(A,\alpha)$ is Galois. We will refer to \eqref{EqGalOth} as the \emph{right Galois map}. 
\end{Rem}

Our aim will be to give a correspondence between particular classes of Galois coactions and homogeneous coactions. Note first that neither the Galois condition nor the homogeneity condition are invariant under general equivariant Morita equivalences. The Galois condition will however be preserved under special types of $H$-equivariant Morita equivalences, closely related to the Morita base change theory for quantum groupoids considered in \cite{Sch02} (in the presence of a unit). We first make the following observation. Recall \cite[Appendix]{VD94} that an algebra homomorphism $f: A\rightarrow B$ is \emph{non-degenerate} if it makes $B$ into a unital left and right $A$-module (in the obvious way). If then $(A,\alpha)$ is a comodule algebra with $H$-coinvariant local units, the embedding 
\[
A^{\alpha}\rightarrow A
\] 
is non-degenerate and we obtain an embedding 
\[
M(A^{\alpha}) \subseteq M(A)
\] 
of multiplier algebras.

\begin{Lem}\label{LemEqMor} 
Let $(H,\Delta)$ be a Hopf algebra, and let $(A,\alpha)$ be a Galois coaction with $H$-coinvariant local units. If $p\in M(A^{\alpha})$ is a full idempotent for $A^{\alpha}$, then $\alpha$ restricts to a Galois coaction on $pAp$ which is equivariantly Morita equivalent to $(A,\alpha)$.
\end{Lem} 

\begin{proof} 
Clearly $\beta = \alpha_{\mid pAp}$ is  a well-defined coaction on $B = pAp$, and $(M,N) = (Ap,pA)$ defines an equivariant Morita context between $A$ and $B$. Clearly $p$ is still a full idempotent in $A$ since 
\[
ApA = AA^{\alpha}pA^{\alpha}A = AA^{\alpha}A = A,
\]
hence the Morita context is strict. 

To see that $\beta$ is Galois, we note that the natural map 
\[
B\underset{B^{\beta}}{\otimes} B \rightarrow pA \underset{A^{\alpha}}{\otimes} Ap
\]
is a linear isomorphism: indeed, upon noting that $B^{\beta} = pA^{\alpha}p = pA^{\alpha}\underset{A^{\alpha}}{\otimes} A^{\alpha}p$, we have 
\begin{align*} 
B\underset{B^{\beta}}{\otimes} B &\cong B\underset{B^{\beta}}{\otimes}B^{\beta}\underset{B^{\beta}}{\otimes} B\\
&= pAp \underset{pA^{\alpha}p}{\otimes} pA^{\alpha}p  \underset{pA^{\alpha}p}{\otimes} pAp \\
&\cong pAp \underset{pA^{\alpha}p}{\otimes} (pA^{\alpha}\underset{A^{\alpha}}{\otimes} A^{\alpha}p)  \underset{pA^{\alpha}p}{\otimes} pAp\\
&\cong (pAp \underset{pA^{\alpha}p}{\otimes} pA^{\alpha})\underset{A^{\alpha}}{\otimes} (A^{\alpha}p \underset{pA^{\alpha}p}{\otimes} pAp) \\
&\cong pA \underset{A^{\alpha}}{\otimes} Ap.
\end{align*}
As the natural map 
\[
pA \underset{A^{\alpha}}{\otimes} Ap \rightarrow A\underset{A^{\alpha}}{\otimes} A
\] 
is injective, and as the Galois map for $\alpha$ thus restricts to an isomorphism 
\[
pA  \underset{A^{\alpha}}{\otimes} Ap \rightarrow B\otimes H,
\] 
we obtain that the Galois map of $\beta$ is an isomorphism as a composition of the isomorphisms 
\[
B \underset{B^{\beta}}{\otimes} B \cong pA\underset{A^{\alpha}}{\otimes} Ap \cong B\otimes H.
\] 
\end{proof}

We now arrive at one of our main theorems.

\begin{Theorem}\label{TheoMain} 
Let $(H,\Delta)$ be a Hopf algebra with invariant functionals, and let $(A,\alpha)$ be an $H$-equivariantly absolutely semisimple algebra with $H$-coinvariant local units. Then $(A,\alpha)$ is equivariantly Morita equivalent to a Galois coaction $(B,\beta)$ with $B^{\beta}\cong k_I$ for some set $I$.
\end{Theorem} 
\begin{proof}
Consider the coaction $((A\#\hat{H})\#H,\widehat{\widehat{\alpha}}_{S^2})$. From Proposition \ref{PropEqMorDual}, we know that this coaction is equivariantly Morita equivalent to $(A,\alpha)$. On the other hand, it is well-known and easily checked that $\widehat{\widehat{\alpha}}$ is Galois, hence $\widehat{\widehat{\alpha}}_{S^2}$ is too. 

Now the algebra of coinvariants for $\widehat{\widehat{\alpha}}_{S^2}$ is $A\#\hat{H}$, which is absolutely semisimple by Theorem \ref{TheoEqMor}. By Lemma \ref{LemFull}, we can pick a full projection $p\in M(A\#\hat{H})$, so that  by Lemma \ref{LemEqMor} we find that  $((A\#\hat{H})\#H,\widehat{\widehat{\alpha}}_{S^2})$ is equivariantly Morita equivalent to a Galois coaction $(B,\beta)$ with $B^{\beta}\cong k_I$ for some set $I$. The theorem now follows by transitivity of equivariant Morita equivalence. 
\end{proof} 

We will reserve a special terminology for the Galois coactions appearing in the previous theorem.

\begin{Def}
Let $I$ be a set, and $(H,\Delta)$ a Hopf algebra. We call \emph{$I$-Galois object} any Galois coaction $(A,\alpha)$ for which $A^{\alpha}\cong k_{I}$.
\end{Def} 

We then have the following partial converse to Theorem \ref{TheoMain}.

\begin{Prop}\label{PropExLoc}
Let $I$ be a set, $(H,\Delta)$ a Hopf algebra with invariant functionals $(\varphi,\psi)$, and $(A,\alpha)$ an $I$-Galois object. Then $(A,\alpha)$ has $H$-coinvariant local units and is  equivariantly absolutely semisimple. 
\end{Prop} 
\begin{proof} 
As clearly $A^{\alpha}$ has local units, it is enough to show that 
\begin{equation}\label{EqLocalUnit}
AA^{\alpha} = A^{\alpha}A = A
\end{equation}
to prove that $A$ has $H$-coinvariant local units. However, choose $h\in H$ with $\varphi(h) = 1$. By the Galois condition we can pick for each $a\in A$ elements $x_i,y_i\in A$ with 
\[
a\otimes h = \sum_i (x_i\otimes 1)\alpha(y_i).
\] 
Applying $\id\otimes \varphi$ and noting that $(\id\otimes \varphi)\alpha(x) \in A^{\alpha}$ for all $x\in A$, it follows that $a \in AA^{\alpha}$.

Similarly, using the bijectivity of \eqref{EqGalOth}, we obtain $A^{\alpha}A = A$. 

To prove that $(A,\alpha)$ is absolutely equivariantly semisimple, it is by Theorem \ref{TheoEqMor} sufficient to prove that $A\#\hat{H}$ is absolutely semisimple. But  since $\alpha$ is Galois, we have by \cite[Theorem 4.4]{VDZ99b} that the algebra $A\# \hat{H}$ is Morita equivalent to the absolutely semisimple algebra $A^{\alpha} =k_I$. 
\end{proof}

If $(H,\Delta)$ is a Hopf algebra and $(A,\alpha)$ a comodule algebra with fixed isomorphism $A^{\alpha} \cong k_{I}$, we will write the minimal idempotent elements of $A^{\alpha} \cong k_I$ as $p_i$. We can then form the vector subspaces 
\[
A_{ij} = p_iAp_j.
\]

\begin{Lem} 
In the above setting, each $A_{ii}$ is a homogeneous $(H,\Delta)$-comodule algebra, and each couple $(A_{ij},A_{ji})$ defines by multiplication an equivariant Morita context between $A_{ii}$ and $A_{jj}$.
\end{Lem} 
\begin{proof} 
Clearly $A_{ii}$ is a unital algebra, with unit $p_i$, to which $\alpha$ restricts. Since 
\[
(p_iAp_i)^{\alpha} = p_i A^{\alpha}p_i = kp_i,
\] 
it is clear that $A_{ii}$ is homogeneous. It is trivially verified that the $(A_{ij},A_{ji})$ then define equivariant Morita contexts.
\end{proof} 

As $A$ has $H$-coinvariant local units, it follows that we have a direct sum decomposition 
\[
A = \oplus_{ij} A_{ij}.
\] 

We now aim to show that in the setting of $I$-Galois objects, the $A_{ij}$ define in fact strict Morita contexts \emph{if} they are non-zero. For this, we first introduce the following terminology.

\begin{Def}\label{DefReyn}
Let $(H,\Delta)$ be a Hopf algebra with invariant functionals $(\varphi,\psi)$, and let $(A,\alpha)$ be a comodule algebra. We define the \emph{Reynolds operator}, or \emph{averaging operator}, as 
\begin{equation}\label{EqDefRey}
\Phi: A \rightarrow  A,\quad a \mapsto (\id\otimes \varphi)\alpha(a).
\end{equation}
\end{Def} 

The following easy lemma was already implicitly used in Proposition \ref{PropExLoc}.

\begin{Lem} 
The image of $\Phi$ lies in $A^{\alpha}$, and $\Phi$ is $A^{\alpha}$-bimodular.
\end{Lem} 

Note however that $\Phi$ does not need to be the identity on $A^{\alpha}$ - this will happen only if $\varphi(1) = 1$. 

For $A$ an $I$-Galois object, it follows that there exist functionals 
\begin{equation}\label{EqVarphii}
\varphi_i: A_{ii} \rightarrow k
\end{equation}
such that 
\begin{equation}\label{EqDefPhi}
\Phi(a) = \sum_i \varphi_{i}(p_iap_i)p_i.
\end{equation}

In the following, we will fix $(H,\Delta)$ with invariant functionals $\varphi,\psi$, and use notation as above.

\begin{Lem}\label{FaithfulPhi} 
Let $(A,\alpha)$ be an $I$-Galois object for $(H,\Delta)$. For each $i,j$, the pairing 
\[
A_{ij}\times A_{ji} \rightarrow A_{ii}, \quad (x,y)\mapsto \varphi_i(xy)
\] 
is non-degenerate.
\end{Lem} 

\begin{proof} 
We are to prove that for any $a\in A_{ji}$ there exists $b\in A_{ij}$ with $\varphi_{i}(ba) \neq 0$, and vice versa. However, fix $a\in A_{ji}$. If $\varphi_i(ba) = 0$ for all $b\in A_{ij}$, it follows that $\Phi(ba) = 0$ for all $b\in A$, and hence 
\[
cb_{(0)}a_{(0)}\varphi(b_{(1)}a_{(1)}),\qquad \forall b,c\in A.
\] 
By the surjectivity of the Galois map, this implies 
\[
ca_{(0)}\varphi(ha_{(1)})=0,\qquad \forall c\in A,h\in H.
\] 
Since $\varphi$ is faithful, and $A$ is non-degenerate, this implies $a_{(0)}\otimes a_{(1)}= 0$, and hence, upon applying the counit, $a=0$.

The non-degeneracy on the other side can be shown similarly using also the surjectivity of the right Galois map $\widetilde{\can}$, see the proof of Proposition \ref{PropExLoc}.
\end{proof} 

The above strong faithfulness of the $\varphi_i$ gives the following important result. 

\begin{Cor}\label{CorEqMorHom} 
Let $A$ be an $I$-Galois object for $(H,\Delta)$. Then either $A_{ij} = A_{ji} = 0$ or $(A_{ij},A_{ji})$ is a strict equivariant Morita context between $A_{ii}$ and $A_{jj}$. 
\end{Cor} 

\begin{proof} 
Assume that $A_{ij}\neq 0$. Then there exists an element $x\in A_{ij}$, and by the previous lemma also an element $y\in A_{ji}$ such that $\varphi_i(xy) = 1$. Hence 
\[
\Phi(xy) = x_{(0)}y_{(0)} \varphi(x_{(1)}y_{(1)}) = p_i.
\]
It follows that $p_i \in A_{ij}A_{ji}$. Similarly, $p_j\in A_{ji}A_{ij}$. This implies immediately that $(A_{ij},A_{ji})$ is a surjective and hence strict Morita context between $A_{ii}$ and $A_{jj}$. 
\end{proof} 

The previous corollary makes it meaningful to introduce the following definition.

\begin{Def} 
Let $(A,\alpha)$ be an $I$-Galois object. We say that $i\sim j$ if $A_{ij}\neq0$. We say that $A$ is \emph{connected} if  $i\sim j$ for all $i,j$. 
\end{Def} 

\begin{Prop}\label{PropConnDecom} 
Let $(A,\alpha)$ be an $I$-Galois object. Then $\sim$ is an equivalence relation on $I$, and $A$ can be written as a direct sum of connected comodule algebras 
\[
A = \underset{x\in I/\sim}{\oplus} A_{x}
\] 
with each $A_x$ a connected $I_x$-Galois object (for $I_x$ the orbit corresponding to $x$). 
\end{Prop} 
\begin{proof} 
Clearly $\sim$ is reflexive, and it is symmetric by Corollary \ref{CorEqMorHom}. Transitivity is also clear since if $i\sim j$ and $j\sim k$ then 
\[
p_k \in A_{kk} = A_{kj}A_{ji}A_{ij}A_{jk}  \subseteq A_{ki}A_{ik}.
\] 
The remainder of the proof is now straightforward.
\end{proof} 

We can hence restrict our focus to connected $I$-Galois objects. Note that we have in this case that for each $i,j,k$ 
\[
A_{ik} = A_{ij}A_{ji}A_{ik} \subseteq A_{ij}A_{jk}  \subseteq A_{ik},
\] 
proving the equality $A_{ij}A_{jk} = A_{ik}$. 

Let us now prepare for our second main theorem. For $(A,\alpha)$ an $I$-Galois object, let us write 
\[
\widetilde{M}_j = \oplus_i A_{ij},\qquad \widetilde{N}_i = \oplus_{j} A_{ij}
\] 
which are respectively equivariant  left and right  unital $A$-modules. 

\begin{Lem}\label{LemMax} 
Let $(A,\alpha)$ be a connected $I$-Galois object. Then each $A_{ii}$ is equivariantly Morita equivalent to $A$. Moreover,
\begin{itemize}
\item the collection $\{\widetilde{M}_j\mid j\in I\}$ is a maximal collection of pairwise non-isomorphic equivariantly simple left $A$-modules, 
\item for each fixed $i$,  the collection $\{A_{ij}\mid j\in I\}$ is a maximal collection of pairwise non-isomorphic equivariantly  simple left $A_{ii}$-modules,
\item the collection $\{\widetilde{N}_i\mid i\in I\}$ is a maximal collection of pairwise non-isomorphic equivariantly simple right $A$-modules, 
\item for each fixed $j$,  the collection $\{A_{ij}\mid i\in I\}$ is a maximal collection of pairwise non-isomorphic equivariantly simple right $A_{jj}$-modules.
\end{itemize}
\end{Lem} 

\begin{proof} 
By  the remark following Proposition \ref{PropConnDecom}, each idempotent $p_i \in A^{\alpha}$ is full in $A$, setting up an equivariant Morita equivalence between $(A,\alpha)$ and $(A_{ii},\alpha)$. In fact, if $M$ is an equivariant $A$-module, then $p_iM \cong p_iA \underset{A}{\otimes} M$ is the associated equivariant $A_{ii}$-module. In particular, $\widetilde{M}_j$ corresponds to $A_{ij}$ under this equivariant Morita equivalence. 

To show that $\widetilde{M}_j$ is equivariantly simple as an $A$-module, it is hence sufficient to show that $A_{ij}$ is equivariantly simple as a left $A_{ii}$-module. Now as $A_{ii}$ is equivariantly Morita equivalent to $A$, it follows from Theorem \ref{TheoEqMor} that $A_{ii}$ is equivariantly absolutely semisimple. Hence $A_{ij} = \oplus_r M_r$ is a direct sum of equivariantly simple $A_{ii}$-modules $M_r$. However, since $A_{ij}$ provides a strict Morita context between $A_{ii}$ and $A_{jj}$, we can identify $A_{jj}$ with the opposite algebra of $\End_{A_{ii}}(A_{ij})$ by right multiplication. Let $p\in A_{jj}$ correspond to a projection on a factor $M_r$, 
\[
A_{ij} \rightarrow M_{r} \subseteq A_{ij},\quad x \mapsto xp.
\]
Then from 
\[
\alpha(xp) = \alpha(x)(p\otimes 1),\qquad \forall x\in A_{ij},
\]
together with $A_{ji}A_{ij} = A_{jj}$, it follows that $p$ is a coinvariant element in $A_{jj}$. As the space of coinvariant elements in $A_{jj}$ is one-dimensional, it follows that $A_{ij}$ itself must already be equivariantly simple.

Similarly, the $\widetilde{M}_j$ are mutually non-isomorphic if we can show that the $A_{ij}$, for $j$ varying, are mutually non-isomorphic. However, by the same argument as above, any isomorphism of equivariant left $A_{ii}$-modules between $A_{ij}$ and $A_{ik}$ must be implemented by right multiplication with a coinvariant element in $A_{jk}$, which can only exist if $j= k$.

To show that $\{A_{ij}\}$ is a maximal collection of equivariantly simple left $A_{ii}$-modules, it suffices to show that $\{\widetilde{M}_j\}$ is a maximal collection of equivariantly simple $A$-modules, or equivalently, of simple $A\#\hat{H}$-modules. Now by \cite[Theorem 4.4]{VDZ99b}, we can turn the vector space $A$ into a strict Morita context between $A\#\hat{H}$ and $k_{I}$, with $A\#\hat{H}$ acting in the natural way. Hence each simple $A\#\hat{H}$-module appears as a direct summand of $A$ by Corollary \ref{CorMorContSem}. By the above however, we have already shown that $A \cong \oplus_j \widetilde{M}_j$ is the direct sum decomposition of $A$ into equivariant simples. 

The statements concerning the equivariant right modules follow by symmetry. 
\end{proof} 

We are now ready to state our second main theorem.

\begin{Theorem} 
Let $(H,\Delta)$ be a Hopf algebra with invariant functionals. Then the assignment 
\[
[(A,\alpha)] \mapsto [(A_{ii},\alpha_{\mid A_{ii}})]
\] 
sets up a one-to-one correspondence between
\begin{itemize}
\item isomorphism classes of comodule algebras $(A,\alpha)$ which are connected $I$-Galois for some set $I$, and 
\item equivariant Morita equivalence classes of homogeneous $H$-equivariantly absolutely semisimple comodule algebras $(C,\gamma)$.
\end{itemize}
\end{Theorem} 
\begin{proof} 
First of all, if $(A,\alpha)$ is a connected $I$-Galois object, we have already shown in the course of Lemma \ref{LemMax} that $A_{ii}$ is indeed homogeneous and $H$-equivariantly absolutely semisimple. By Corollary \ref{CorEqMorHom}, the $A_{ii}$ do not depend on $i$ up to equivariant Morita equivalence. This shows that the correspondence in the theorem is well-defined. We are left to show that the correspondence is bijective.

Let $(C,\gamma)$ be a homogeneous, $H$-equivariantly absolutely semisimple comodule algebra. By Theorem \ref{TheoMain}, $(C,\gamma)$ is equivariantly Morita equivalent to an $I$-Galois object $(A,\alpha)$ for some set $I$. Let $A =\oplus_{I_x\in I/\sim} A_x$ be the direct sum decomposition of $A$ into connected $I_x$-Galois objects $A_x$. Choosing in each $I_x$ a fixed element $i_x\in I_x$, the reasoning in the first paragraph gives that $(A,\alpha)$, and hence $(C,\gamma)$, is equivariantly Morita equivalent to $\oplus_x A_{i_x,i_x}$. However, this implies that $C \cong \oplus_x C_{x}$ as a direct sum of comodule algebras. As $C^{\gamma} = k1$, only one component can appear. It follows that $I/\sim$ is a singleton, and hence $(A,\alpha)$ is connected. This proves surjectivity of the correspondence. 

To see that the correspondence is injective, let $A$ be a connected $I$-Galois object, and fix $i\in I$. Write $(C,\gamma) = (A_{ii},\alpha_{\mid A_{ii}})$. Write $(M_j,\beta_j) = (A_{ji},\alpha_{\mid A_{ji}})$. Then by Lemma \ref{LemMax} the $\{M_j\}_{j\in I}$ form a maximal collection of mutually non-isomorphic simple $H$-equivariant right $C$-modules. Moreover, putting 
\[
B_{kj} = \Mor_C(M_j,M_k),
\]
we have a natural vector space isomorphism
\[
A_{kj} \cong B_{kj}
\]
by left multiplication, as the $A_{kj}$ are strict Morita contexts. This sets up an isomorphism $A \cong \oplus_{kj} B_{kj}$ of algebras. The ensuing comodule structures $\beta_{kj}$ on the $B_{kj}$ are moreover uniquely determined by the condition 
\[
\beta_{kj}(x)\beta_j(m)  =\beta_{k}(xm),\qquad x\in B_{kj},m\in M_j.
\]
It follows that up to isomorphism, $(A,\alpha)$ can be constructed from the homogeneous comodule algebra $(C,\gamma)$, proving injectivity of our correspondence. 
\end{proof}

To end, let us present a criterion for absolute equivariant semisimplicity. Note first that if $(A,\alpha)$ is a comodule algebra for $(H,\Delta)$, and $(V,\delta)$ a comodule for $(H,\Delta)$, then we can form the equivariant comodule algebra $A\otimes V$ upon which $A$ acts by multiplication in the first tensorand, and $(H,\Delta)$ coacts by the tensor product comodule structure. 

\begin{Lem}
Let $(H,\Delta)$ be a Hopf algebra, and let $(A,\alpha)$ be a unital algebra. Then a coaction $(A,\alpha)$ is absolutely equivariantly semisimple if and only if the $H$-equivariant $A$-modules $A\otimes V$ are  equivariantly absolutely semisimple for all finite-dimensional comodules $(V,\delta)$.
\end{Lem}
The unitality of $(A,\alpha)$ can be relaxed, but we will not need this generalization for our application. 
\begin{proof}
As in the non-equivariant case, it is not hard to see that an $H$-equivariant module is equivariantly absolutely semisimple if and only if it is a (not necessarily direct) sum of equivariantly absolutely simple modules. As any $(H,\Delta)$-comodule is a sum of finite-dimensional comodules, it is hence sufficient to show equivariant absolute semisimplicity of any $H$-equivariant $A$-module $W$ which is generated, as an $A$-module, by a finite-dimensional $H$-comodule $V \subseteq W$. However, in this case we obtain $W$ as a quotient of $A\otimes V$ by 
\[
A \otimes V \rightarrow W,\quad a\otimes v \mapsto av.
\]
Since equivariant absolute semisimplicity of equivariant modules passes to quotients, this proves the lemma. 
\end{proof}

Assume now that $(H,\Delta)$ is a Hopf $*$-algebra (over the ground field $k=\C$) with a \emph{positive} invariant integral \cite{VD94}, so
\[
\varphi(h^*h) \geq 0,\qquad \forall h\in H.
\]
Assume that $A$ is a unital $*$-algebra with a homogeneous ($*$-preserving) coaction $\alpha$. Then by homogeneity there exists a unique functional $\varphi_A$ on $A$ such that
\[
\varphi_A(a)1_A = (\id\otimes \varphi)\alpha(a),\qquad \forall a\in A.
\]
Let us say that $\varphi_A$ is positive if $\varphi_A(a^*a)\geq 0$ for all $a\in A$.

\begin{Prop}
Let $(H,\Delta)$ be a Hopf $*$-algebra (over $\C$) with a positive invariant integral $\varphi$. Let $(A,\alpha)$ be a homogeneous coaction on a unital $*$-algebra $A$ such that $\varphi_A$ is positive. Then $(A,\alpha)$ is absolutely equivariantly semisimple.
\end{Prop}
\begin{proof}
Let us give a brief sketch of the proof, see also \cite[Section 6]{DC17}. 

As $\C$ is algebraically closed, it is sufficient to show equivariant semisimplicity. By the previous lemma, it suffices to show equivariant semisimplicity of the $A\otimes V$ for $(V,\delta)$ a finite-dimensional comodule. Now $A$ can be endowed with a pre-Hilbert space structure by 
\[
\langle a,b\rangle = \varphi_A(a^*b).
\]
On the other hand, also $V$ can be endowed with a Hilbert space structure which is preserved by $\delta$. By \cite{Boc95}, we know that  the isotypical components of $A$ are finite-dimensional and orthogonal \cite{Boc95}. Hence the same holds for $A\otimes V$. It follows immediately that each equivariant submodule of $A\otimes V$ must be complemented canonically by its orthogonal complement, proving equivariant semisimplicity of $A\otimes V$.
\end{proof}

Note that in the setting of C$^*$-algebras, equivariant semisimplicity for homogeneous actions of compact quantum groups was shown first in \cite{Ver02}.

\section{$I$-Galois objects for Hopf algebras with invariant functionals}

In this final section, we consider in some more detail the notion of $I$-Galois object, complementing results from \cite{DC09}. 

Fix $(H,\Delta)$ a Hopf algebra with invariant functionals $\varphi,\psi = \varphi \circ S$. We use the notation introduced in Proposition \ref{PropHopfInv}. We also fix an $I$-Galois object $(A,\alpha)$. We write $p_i$ for the minimal idempotents in $A^{\alpha}\cong k_I$. We write $\splitting$ for the \emph{splitting map}
\begin{equation}\label{EqSplitting}
\splitting: A \otimes_{k_I}A \rightarrow A\otimes A,\quad a\otimes a'\mapsto \sum_i ap_i \otimes p_ia',
\end{equation}
where the final expression in \eqref{EqSplitting} is in fact a finite sum by the existence of coinvariant local units, see in particular \eqref{EqLocalUnit}. Note that $\splitting$ sets up a linear isomorphism with the subspace $\oplus_{i,j,k} (A_{ij}\otimes A_{jk})$ of $A\otimes A$.

\subsection{Existence of invariant functionals}

Recall the map $\Phi$ introduced in Definition \ref{DefReyn}.

\begin{Lem}\label{LemDelInv}
The map $\Phi: A \rightarrow A^{\alpha}$ satisfies
\[
(\Phi \otimes \id)\alpha(a) = \Phi(a)\otimes \delta,\qquad \forall a\in A. 
\]
\end{Lem}
\begin{proof}
This follows immediately from the identity
\[
(\varphi\otimes \id)\Delta(h) = \varphi(h)\delta,\qquad h\in H,
\]
see \cite[Proposition 3.8]{VD98}.
\end{proof}

Let us denote 
\[
\mathscr{S}: A^{\alpha} \rightarrow k,\quad \sum_i a_i p_i \mapsto \sum_i a_i,
\]
and
\[
\varphi_A = \mathscr{S} \circ \Phi: A \rightarrow k.
\]
By abuse of notation, we will also write the natural extension of the maps $\varphi_i$ from \eqref{EqVarphii} to $A$ as $\varphi_i$, so that
\begin{equation}\label{EqDefPhii}
\varphi_i: A \rightarrow k,\quad a \mapsto \varphi_i(a) = \varphi_i(p_iap_i).
\end{equation}
Then we have
\[
\varphi_A(a) = \sum_i \varphi_i(a),\qquad a\in A.
\]
Note again that the sum in the right hand side is finite, as $\varphi_i(a)=0$ for all but finitely many $i \in I$. By Lemma \ref{FaithfulPhi} and Lemma \ref{LemDelInv}, we obtain the following.

\begin{Prop}
The functional $\varphi_A$ is faithful, i.e. the pairing
\[
A \times A \rightarrow k,\qquad (a,a') \mapsto \varphi_A(aa')
\]
is non-degenerate. Moreover, the functional $\varphi_A: A \rightarrow k$ is $\delta$-invariant:
\[
(\varphi_A\otimes \id)\alpha(a) = \varphi_A(a)\delta,\qquad \forall a\in A.
\] 
\end{Prop}

In particular, also the $A^{\alpha}$-valued pairing
\[
A \times A \rightarrow A^{\alpha},\quad (a,a')\mapsto \Phi(aa')
\]
is non-degenerate.

Having the $\delta$-invariant functional $\varphi_A$ on $A$, it is natural to ask whether there exists also an \emph{invariant} functional. It turns out that this is indeed the case, but this is not immediately obvious. If $\psi: A \rightarrow k$ is any functional, let us write
\[
\psi_{ij}: A \rightarrow k,\quad a \mapsto \psi(p_iap_j).
\]
for its components. We also write
\[
\psi_{i-} = \psi(p_i-),\qquad \psi_{-j} = \psi(-p_j).
\]

\begin{Def}
A functional $\psi: A \rightarrow k$ is called \emph{invariant} if 
\[
(\psi\otimes \id)\alpha(x) = \psi(x)1_H,\qquad \forall x\in A. 
\]
It is called \emph{left complete} if $\psi_{i-}\neq 0$ for all $i$, and \emph{right complete} if $\psi_{-j}\neq 0$ for all $j$. 
\end{Def}

\begin{Lem}\label{LemInvpsi}
There exists a left complete invariant functional $\psi_{A}$. Similarly, there exists a right complete invariant functional $\psi_A'$. 
\end{Lem}

\begin{proof}
By surjectivity of the Galois maps
\[
a \otimes a' \mapsto (a\otimes 1_H)\alpha(a'),\qquad a\otimes a' \mapsto \alpha(a)(a'\otimes 1_H),
\]
together with the fact that none of the $\varphi_i$ is zero, it follows that for each index $i\in I$ there exists an element $y_i\in A$ such that 
\[
\psi_{i-}: A \mapsto k,\quad x \mapsto\varphi_i(x_{(0)}y_i)\psi(x_{(1)})
\]
is not the zero map, and that there exists for any index $j\in I$ an element $z_j\in A$ such that 
\[
\psi_{-j}': A \mapsto k,\quad w \mapsto \varphi_j(z_jw_{(0)})\psi(w_{(1)})
\]
is not the zero map. 

A direct calculation for any $x\in A$, $i\in I$ shows that:
\begin{align*}
\psi_{i-}(x_{(0)})x_{(1)}&=\varphi_{i}(x_{(0)}y_i)\psi(x_{(1)})x_{(2)}\\
&= \varphi_i(x_{(0)}y_i)\psi(x_{(1)})1_H\\
&= \psi_{i-}(x)1_H.
\end{align*}
Replacing $\psi_{i-}$ by $\psi_{-j}'$ proves the same result for $\psi_{-j}'$. Put now
\begin{equation}\label{EqDefPsiA}
\psi_A(a) = \sum_i \psi_{i-}(a),\qquad \psi_A' = \sum_j \psi_{-j}'(a),
\end{equation}
where one notes that the sums are in fact finite. Then $\psi_A$ and $\psi_A'$ define respectively left and right complete invariant functionals.
\end{proof}

\begin{Lem}\label{LemTheta}
Let $\psi_A$ be an invariant functional. Then there exists a unique linear map $\theta:A\to A$ such that for all $x,w\in A$
\[
\varphi_A\left(x\theta(w)\right) = \psi_{A}\left(xw\right).
\]
If $\psi_A$ is left complete, then $\theta$ is surjective. Similarly, if $\psi_A'$ is an invariant functional, there exists a unique linear map $\theta': A\to A$ such that for all $x,y\in A$
\[
\varphi_A\left(\theta'(x)y\right) = \psi_A'\left(xy\right),
\]
and $\theta'$ is surjective if $\psi_A'$ is right complete. 
\end{Lem}
\begin{proof}
We explicitly construct the map $\theta$ here. The proof of the existence and surjectivity of $\theta'$ is analogous. Note that the uniqueness of both maps is immediate from the faithfulness of $\varphi_A$. 

We first make the following computation. Let $w\in A,h\in H$, and put 
\[
(\splitting \circ \widetilde{\can}^{-1})(\alpha(w)(1\otimes h)) = \sum_{r = 1}^n y_r \otimes z_r, 
\]
where $\widetilde{\can}$ is the right Galois map \eqref{EqGalOth}. Then we have 
\begin{equation}\label{EqOneMore}
\alpha(w)(1\otimes h) = \sum_{r = 1}^n \alpha(y_r)(z_r\otimes 1),
\end{equation}
and 
\begin{equation}\label{Eqyp}
y_r p_i \otimes z_r = y_r\otimes p_iz_r,\qquad \textrm{for all }r,i.
\end{equation}
By invariance of $\psi_{i-}$, we then obtain for all $x\in A$ and $i\in I$ that
\begin{eqnarray*}
\psi_{i-}(xw)\varphi(h) &= &(\psi_{i-}\otimes \varphi)(\alpha(xw)(1\otimes h)) \\
&\underset{\eqref{EqOneMore}}{=}&  \sum_{r = 1}^n (\psi_{i-}\otimes \varphi)(\alpha(xy_r)(z_r\otimes 1))\\
&\underset{\eqref{EqDefRey}}{=}& \sum_{r = 1}^n  \psi_{i-}(\Phi(xy_r)z_r)\\
&\underset{\eqref{EqDefPhi},\eqref{EqDefPhii}}{=}& \sum_{r = 1}^n  \varphi_i(xy_r)\psi_{i-}(z_r)\\
&=&\varphi_i\left(x\left(\sum_{r = 1}^{n}\psi_{i-}(z_r)y_r\right)\right)\\
&\underset{\eqref{EqDefPsiA}}{=}& \varphi_i\left(x\left(\sum_{r = 1}^{n}\psi_A(p_iz_r)y_r\right)\right)\\
&\underset{\eqref{Eqyp}}{=}& \varphi_i\left(x\left(\sum_{r = 1}^{n}\psi_A(z_r)y_rp_i\right)\right)\\
&=& \varphi_i\left(x\left(\sum_{r = 1}^{n}\psi_A(z_r)y_r\right)\right)
\end{eqnarray*}
Putting 
\[
\theta_h: A \rightarrow A,\quad w \mapsto ((\id\otimes \psi_A)\circ \splitting \circ \widetilde{\can}^{-1})(\alpha(w)(1\otimes h)),
\] 
we hence see that 
\[
\varphi(h)\psi_A(xw) = \varphi_A(x\theta_h(w)),\qquad \forall x,w\in A,h\in H.
\]
In particular, fix $h'\in H$ with $\varphi(h') = 1$, and put $\theta = \theta_{h'}$. Then we see that
\[
\psi_A(xw) = \varphi_A(x\theta(w)),\qquad \forall x,w\in A.
\]
Moreover, since
\[
\varphi(h)\varphi_A(x\theta(w)) = \varphi(h)\psi_A(xw) = \varphi_A(x\theta_h(w))
\]
for all $h\in H$ and $x,w\in A$, faithfulness of $\varphi_A$ entails that
\begin{equation}\label{EqHulp1}
\theta_h(w) = \varphi(h)\theta(w),\qquad \forall h\in H,w\in A.
\end{equation}
Let us now finally show that $\theta$ is surjective in case $\psi_A$ is left complete. Consider $i,j\in I$ and $y\in A_{ij}$. As $\psi_{j-}\neq 0$, we can find $k\in I$ and $z\in A_{jk}$ with $\psi_A(z) = \psi_{j-}(z) = 1$. On the other hand, we can find finitely many $w_r\in A$ and $h_r\in H$ such that
\[
\alpha(y)(z\otimes 1) = \sum_r\alpha(w_r)(1\otimes h_r).
\]
By bijectivity of $\widetilde{\can}$ and the form of the range of $\splitting$, it follows that
\begin{equation}\label{EqHulp2}
\sum_r (\splitting \circ \widetilde{\can}^{-1})(\alpha(w_r)(1\otimes h_r)) = y \otimes z, 
\end{equation}
and hence
\begin{eqnarray*}
\theta\left(\sum_r \varphi(h_r)w_r\right) &\underset{\eqref{EqHulp1}}{=}& \sum_r \theta_{h_r}(w_r)\\
&\underset{\eqref{EqHulp2}}{=}& \psi_A(z)y \\
&=& y.
\end{eqnarray*} 
\end{proof}

\begin{Cor}
Let $\psi_A,\psi_A'$ be invariant functionals. Then for all $i\in I$
\[
\{\psi_{i-}(-y)|y\in A\} \subseteq \{\varphi_i(-w)|w\in A\},\qquad \{\psi_{-j}'(x-)\mid x\in A\} \subseteq \{\varphi_i(z-)\mid z\in A\}.
\]
If moreover $\psi_A$ is left complete, resp. $\psi_A'$ right complete, then these inclusions become equalities. 
\end{Cor}

In what follows, we will fix resp. left and right complete invariant functionals $\psi_A,\psi_A'$. We want to prove a uniqueness result for the $\psi_{i-}$ and $\psi_{-j}'$. Consider
\[
\psi_{ij}: A \rightarrow k,\quad a \mapsto \psi_A(p_iap_j),\qquad \psi_{ij}': A \rightarrow k,\quad a \mapsto \psi_A'(p_iap_j).
\]
\begin{Lem}\label{LemScale}
Let $\psi_A,\psi_A'$ be invariant functionals. If $\psi_A$ is left complete, there exists for each $j\in I$ a constant $c_j \in k$ such that for all $i\in I,x\in A$ 
\[
\psi_{ij}'(x) = c_j\psi_{ij}(x).
\]
Similarly, if $\psi_A'$ is right complete, there exists for each $i\in I$ a constant $d_i \in k$ such that for all $j\in I,x\in A$ 
\[
\psi_{ij}(x) = d_i\psi_{ij}'(x).
\] 
\end{Lem}
\begin{proof}
We will prove the existence of the $c_j$. The existence of the $d_i$ follows again similarly.

By Lemma \ref{LemTheta}, there exists for each $j\in I$ an element $p_j' \in A$ such that 
\begin{equation}\label{def ptilde}
\psi_{ij}' = \psi_{i-}(-p_j'). 
\end{equation}
Fix now $j\in I$, and fix an element $y\in A_{jj}$ such that $\varphi_j(y)=1$. Let $x\in A$. Then there must exist finitely many elements $h_1,\dots,h_n\in H$ and $a_1,\dots,a_n\in A$ such that
\begin{equation}\label{vgl 1}
(x\otimes 1)\alpha(y) = \sum_{r = 1}^{n}(1\otimes h_r)\alpha(a_r).
\end{equation}
Using that $\varphi_j(y) = 1$, we hence find that
\begin{eqnarray*}
\psi_{ij}'(x) &=& \psi_{A}'(p_ix\Phi(y)p_j)\\
&=& (\psi_{ij}'\otimes \varphi)\left((p_ix\otimes 1)\alpha(y p_j)\right)\\
&\underset{\eqref{vgl 1}}{=}& (\psi_{ij}'\otimes \varphi)\left(\sum_{r = 1}^{n}(1\otimes h_r)\alpha(p_ia_r p_j)\right)\\
&=& \sum_{r = 1}^{n}\varphi(h_r)\psi_{ij}'(a_r)\\
&\underset{\eqref{def ptilde}}{=}& \sum_{r= 1}^{n}\varphi(h_r)\psi_{i-}(a_r p_j')\\
&=&(\psi_{i-}\otimes \varphi)\left(\sum_{r = 1}^{n}{(1\otimes h_r)\alpha(a_r p_j')}\right)\\
&\underset{\eqref{vgl 1}}{=}& (\psi_{i-}\otimes \varphi)((x\otimes 1)\alpha(y p_j'))\\
&=& \psi_{i-}(x\Phi(y p_j'))\\
&=& \sum_{k\in I} \varphi_{k}(y p_j')\psi_{i-}(x p_{k})\\
&=& \varphi_{j}(yp_j') \psi_{ij}(x),
\end{eqnarray*}
where in the last step we used that $y \in A_{jj}$. It follows that
\[
\psi_{ij}'(x) = \varphi_j(yp_j') \psi_{ij}.
\]
\end{proof}

\begin{Cor}
The space of invariant functionals on $A_{ij}$ is at most one-dimensional. 
\end{Cor}

\begin{Theorem}\label{TheoUnique}
Let $\psi_A$ be a left complete invariant functional. Then there exists a bijection $\mu: I \rightarrow I$ such that $\psi_{ij} \neq 0$ iff $j = \mu(i)$. Moreover, if $\widetilde{\psi}$ is any other invariant functional on $A$, there exist $c_i \in k$ such that 
\[
\widetilde{\psi}_{ij} = \delta_{j,\mu(i)} c_i \psi_{ij}.
\]
\end{Theorem}
\begin{proof}
Pick an arbitrary $\mu: I \rightarrow I$ such that $\psi_{i,\mu(i)}\neq 0$, and put
\[
\widetilde{\psi}_A = \sum_i \psi_{i,\mu(i)}.
\]
Then one can apply Lemma \ref{LemScale} with $\psi_A$ replaced by $\widetilde{\psi}_A$ and $\psi_A'$ replaced by $\psi_A$, to conclude that 
\[
\psi_{ij} = c_j \widetilde{\psi}_{ij},
\]
for certain $c_j \in k$. In particular, $\psi_{ij}= 0$ unless $j= \mu(i)$. 

Let us show that $\mu$ is bijective. If $\mu$ were not surjective, there exists $k\in I$ with $\psi_{-k} = 0$. Pick however $i\in I$ with $\psi_{ik}' \neq 0$ for some right complete invariant functional $\psi_A'$. Then from Lemma \ref{LemScale}, we obtain that $\psi_{ik}'$ is a multiple of $\psi_{ik} =0$, a contradiction. To see that $\mu$ is injective, note that the same argument as above implies the existence, for any right complete invariant functional $\psi_A'$, of a $\nu: I \rightarrow I$ such that $\psi_{ij}'\neq 0$ unless $i = \nu(j)$. From Lemma \ref{LemScale}, we conclude that there exist constants $d_i \in k$ such that
\[
\psi_{i,\mu(i)}(x) = d_i \psi_{i,\mu(i)}'(x) = d_i \delta_{i,\nu(\mu(i))} \psi_{i,\mu(i)}'(x).
\]
It follows that necessarily $i = \nu(\mu(i))$, hence $\mu$ is injective. 

The remaining uniqueness claim in the theorem is now clear.
\end{proof}

By considering the analogue of the above theorem for right complete invariant functionals, we deduce the following corollary.

\begin{Cor}
An invariant functional is left complete if and only if it is right complete. 
\end{Cor}

We will in the following hence speak of \emph{complete} invariant functionals. In the following, we will fix a complete invariant functional $\psi_A$ and associated bijective function $\mu: I \rightarrow I$ such that 
\[
\psi_{ij} = \delta_{i,\mu(i)}\psi_{i,\mu(i)},\qquad \psi_{i,\mu(i)}\neq 0.
\]
We will write $\kappa = \mu^{-1}$. Note that $\mu$ depends only on $(A,\alpha)$, as any other complete invariant functional $\psi_A'$ must by Theorem \ref{TheoUnique} be of the form 
\[
\psi_{ij}'= \delta_{j,\mu(i)}c_i \psi_{ij},\qquad c_i \in k^{\times}.
\]

\begin{Rem}
We do not know if necessarily $\mu  = \id$. Indeed, although we know of no counterexamples, we could not find an argument enforcing this condition. On the other hand if $(H,\Delta)$ has invariant \emph{integrals}, the identity $\mu = \id$ becomes immediate, as then $\varphi_i(p_i) = 1$ and hence $\psi_{ii}(x) = \varphi_i(x_{(0)})\psi(x_{(1)})$ is an invariant functional not vanishing on $p_i \in A_{ii}$.
\end{Rem}

\begin{Lem}
The functional $\psi_A$ is faithful.
\end{Lem}
\begin{proof}
Pick $x\in A$ with $\psi_A(x-) =0$. From the construction in Lemma \ref{LemInvpsi}, we know that for any $y\in A$ the map
\[
w \mapsto (\varphi_A\otimes \psi)(\alpha(w)(y\otimes 1))
\]
is invariant. From the uniqueness in Theorem \ref{TheoUnique}, it follows that then necessarily 
\[
(\varphi_A\otimes \psi)(\alpha(xw)(y\otimes 1)) = 0,\qquad \forall w,y\in A.
\]
By surjectivity of the Galois map, this implies
\[
(\varphi_A\otimes \psi)(\alpha(x)(y\otimes h)) = 0,\qquad \forall y\in A,h\in H.
\]
As $\varphi_A$ and $\psi$ are faithful, we deduce $\alpha(x)=0$, and hence $x= 0$. 

The proof that $\psi_A(-y) =0$ implies $y=0$ is similar. 
\end{proof}

\begin{Cor}
The maps $\theta, \theta'$ in Lemma \ref{LemTheta} are bijective, and induce linear isomorphisms
\[
\theta: A_{ij} \mapsto A_{i\kappa(j)}, \qquad \theta': A_{ij} \mapsto A_{\mu(i),j}.
\]
\end{Cor}
\begin{proof}
The surjectivity of the maps was already proven in Lemma \ref{LemTheta}.
For injectivity, suppose $\theta(y) = 0$.
Then for all $x\in A$
\[
0 = \varphi_A(x\theta(y)) = \psi_A(xy).
\] 
This implies that $y = 0$, by faithfulness of $\psi_A$. The proof for $\theta'$ is completely analogous. The behaviour of $\theta,\theta'$ with respect to the components is clear from their definition and faithfulness of $\psi_A$.
\end{proof}

\begin{Def}
For each $i\in I$, we define 
\[
\delta_i =\theta(p_i) \in A_{i,\kappa(i)},\qquad \delta_i' = \theta'(p_i) \in A_{\mu(i),i}.
\]
We can then make sense of 
\[
\delta_A = \sum_i \delta_i \in M(A),\qquad \delta_A' = \sum_i \delta_i'\in M(A).
\]
We call $\delta_A$ the \emph{modular element} of $A$, and $\delta_A'$ the \emph{associated modular element}.
\end{Def}

We obtain by construction that 
\[
\psi_A(x) = \varphi_A(x\delta_A) = \varphi_A(\delta_A'x),\qquad \forall x\in A,
\]
so that 
\[
\theta(x) = x\delta_A,\qquad \theta'(x) = \delta_A'x.
\]
Bijectivity of $\theta,\theta'$ guarantees then that $\delta_A,\delta_A'$ are invertible.

\begin{Lem}
The elements $\delta_A,\delta_A'$ satisfy
\[
\alpha(\delta_A) = \delta_A\otimes \delta,\qquad \alpha(\delta_A') = \delta_A'\otimes \delta.
\]
\end{Lem}
\begin{proof}
For $x\in A,h\in H$ we compute
\begin{eqnarray*}
(\varphi_A\otimes \varphi)((1\otimes h)\alpha(x\delta_A)) &=& \varphi_A(x\delta_A)\varphi(h\delta) \\&=&  \psi_A(x)\psi(h)\\  
&=& (\psi_A\otimes \psi)((1\otimes h)\alpha(x))\\
&=& (\varphi_A\otimes \varphi)((1\otimes h)\alpha(x)(\delta_A\otimes \delta)).
\end{eqnarray*}
As $\varphi_A,\varphi$ are faithful, and the map 
\[
h\otimes x \rightarrow (1\otimes h)\alpha(x)
\]
is surjective onto $A\otimes H$, it follows that $\alpha(\delta_A) = \delta_A\otimes \delta$. The proof for $\delta_A'$ is similar.
\end{proof}
\begin{Cor}
There exist unique scalars $\nu_i$, only depending on $(A,\alpha)$, such that
\[
\delta_i' = \nu_i \delta_{\mu(i)}. 
\]
\end{Cor}
\begin{proof}
By the previous lemma, the components $p_i\delta_A^{-1}\delta_A' p_i$ are coinvariant, and hence of the form $\nu_i p_i$ for some $\nu_i \in k^{\times}$. The independence of $\nu$ with respect to $\psi_A$ follows from Theorem \ref{TheoUnique}.
\end{proof}

\subsection{The modular automorphisms}

We still have fixed an $I$-Galois object $(A,\alpha)$. In this section, our aim is to construct an automorphism of the algebra $A$, satisfying the following property for all $a,b\in A$:
$$\varphi_A(ab) = \varphi_A(b\sigma(a)).$$

In the following, let us denote
\[
\beta_i: H \rightarrow A\otimes A,\qquad h \mapsto (\splitting \circ \can^{-1})(p_i\otimes h) = h^{[1;i]}\otimes h^{[2;i]}.
\]
\begin{Lem}
For all $h\in H, x\in A, i\in I$ the following relations holds:
\begin{eqnarray}
h^{[1;i]}\Phi(h^{[2;i]}x) &=& \varphi(hx_{(1)})p_ix_{(0)},\label{eig1}\\
\Phi(xh^{[1;i]})h^{[2;i]} &=& \varphi(x_{(1)}S(h))x_{(0)}p_i.\label{eig2}
\end{eqnarray}
\end{Lem}
\begin{proof}
Take any $h\in H, x\in A, i\in I$ and calculate:
\begin{eqnarray*}
h^{[1;i]}\Phi(h^{[2;i]}x) &=& h^{[1;i]}\left((\operatorname{id}\otimes\varphi) \alpha(h^{[2;i]}x)\right)\\
&=& (\operatorname{id}\otimes\varphi)\left(\operatorname{can}(h^{[1;i]}\otimes h^{[2;i]})\alpha(x)\right)\\
&=& ((\operatorname{id}\otimes \varphi)(p_ix_{(0)}\otimes hx_{(1)})\\
&=& \varphi(hx_{(1)})p_ix_{(0)}.
\end{eqnarray*}
The second equality follows similarly, using the easily derived formula
\[
h^{\lbrack 1;i\rbrack}_{\;(0)} h^{\lbrack 2;i\rbrack} \otimes h^{\lbrack 1;i\rbrack}_{\;(1)} = p_i \otimes S(h).
\]
\end{proof}
Note now that $$h^{[1;i]}\otimes h^{[2;i]}\in \bigoplus_{j\in I} A_{ij}\otimes A_{ji},$$ since $\can$ commutes with left multiplication with $p_i$ in the first leg, and right multiplication with $p_i$ in the second leg. 
\begin{Prop}
There exists an automorphism $$\sigma_A:A\to A,$$
such that for all $x,y\in A$
\[
\varphi_A(xy) = \varphi_A(y\sigma_A(x)).
\]
\end{Prop}
\begin{proof}
First, fix an arbitrary $h\in H$ and set $g = (S^{-1}\sigma)(h)$, where $\sigma$ is the modular automorphism of $(H,\Delta)$. Pick $p,q\in A$, and put
\[
x= \sum_i \varphi_A(ph^{[1;i]}q)h^{[2;i]},\qquad x' = \sum_i g^{[1;i]}\varphi_A(pg^{[2;i]}q).
\]
Then we compute for any $y\in A$ that 
\begin{eqnarray*}
\varphi_A(yx') &=& \sum_{i,j\in I}\varphi_i(yg^{[1;j]})\varphi_{A}(pg^{[2;j]}q)\\
&=& \sum_j \varphi_A(p \Phi(yg^{[1;j]})g^{[2;j]}q)\\
&\underset{\eqref{eig2}}{=}&\sum_{j\in I}\varphi(y_{(1)}\sigma(h))\varphi_{A}(py_{(0)}p_jq)\\
&=& \varphi(y_{(1)}\sigma(h))\varphi_{A}(py_{(0)}q)\\
&=&\varphi(hy_{(1)})\varphi_{A}(py_{(0)}q)\\
&=& \sum_i \varphi(hy_{(1)})\varphi_{A}(pp_iy_{(0)}q)\\
&=&\sum_{i,j\in I}\varphi_{j}(h^{[2;i]}y)\varphi_{A}(ph^{[1;i]}q)\\
&=&\varphi_{A}(xy).
\end{eqnarray*}

We want to prove now that we can find such $x'\in A$ for any $x\in A$. It suffices to show that the elements $x$ as constructed above span $A$. However, assume $x\in A_{i-}$, and pick $z \in A_{ii}$ with $\varphi_i(z) = 1$. Write
\[
\can(z \otimes x) = \sum_{r=1}^n y_r\otimes h_r.
\]
Pick $e\in A^{\alpha}$ with $ey = y$. 
Then 
\begin{eqnarray*}
x &=& \varphi_A(z)x \\
&=& \sum_{r=1}^n\sum_{i} \varphi_A(y_rh_r^{[1;i]})h_r^{[2;i]}\\
&=& \sum_{r=1}^n\sum_{i} \varphi_A(y_rh_r^{[1;i]}e)h_r^{[2;i]}.
\end{eqnarray*}
It follows that $x$ is indeed of the required form. 

Faithfulness of $\varphi_A$ now guarantees that the map
\[
\sigma_A: A \rightarrow A,\quad x \mapsto x'
\]
is well-defined, linear and injective. A similar argument as in the previous paragraph shows that $\sigma_A$ is also surjective. Finally, the faithfulness of $\varphi_A$ and the defining property of $\sigma_A$ also guarantee immediately that $\sigma_A(xy) = \sigma_A(x)\sigma_A(y)$ for all $x,y\in A$.
\end{proof}

\begin{Def}
We call $\sigma_A$ the \emph{Nakayama automorphism}, or \emph{modular automorphism} of $\varphi_A$.
\end{Def}

\begin{Rem}
There also exists a map which plays the same role for an invariant complete functional $\psi_A$.
Take two elements $x,y\in A$, then 
\[
\psi_A(xy) = \varphi_A(xy\delta_A) = \varphi_A(y\delta_A\sigma_A(x)\delta_A^{-1}\delta_A) = \psi_A(y\delta_A\sigma_A(x)\delta_A^{-1}).
\]
So the associated modular automorphism $\sigma_A$ is defined as
$$\sigma_A': A \rightarrow A,\quad x\mapsto \delta_A\sigma_A(x)\delta_A^{-1}.$$
\end{Rem}


\begin{thebibliography}{00}
\bibitem[Abr83]{Abr83} G.D. Abrams, Morita equivalences for rings with local units, \emph{Comm. Algebra} \textbf{11} (8) (1983), 801--837. 
\bibitem[AM87]{AM87} P.N. \'{A}nh and L. M\'{a}rki, Morita equivalence for rings without identity, \emph{Tsukuba J. Math.} \textbf{11} (1) (1987), 1--16.
\bibitem[BDCH17]{BDCH17} P. F. Baum, K. De Commer and P. M. Hajac, Free actions of compact quantum group on unital C$^*$-
algebras, \emph{Doc. Math.} \textbf{22} (2017), 825--849.
\bibitem[BDR97]{BDR97} M. Beattie, S. Dascalescu and S. Raianu, Galois extensions for co-Frobenius Hopf algebras, \emph{J. Algebra} \textbf{198} (1997), 164--183.
\bibitem[Beh73]{Beh73} E.-A. Behrens, Ring Theory, \emph{Academic Press}, New York (1972).
\bibitem[Boc95]{Boc95} F. Boca, Ergodic actions of compact matrix pseudogroups on C$^*$-algebras, in \emph{Recent Advances in Operator Algebras}, \emph{Ast\'{e}risque} \textbf{232} (1995), 93--110.
\bibitem[Cohn89]{Cohn89} P.M. Cohn, Algebra, Volume 2 (Second edition), John Wiley \& Sons Ltd., Chichester (1989), xvi+428 pp. 
\bibitem[DC09]{DC09} K. De Commer, Galois objects for algebraic quantum groups, \emph{J. Algebra} \textbf{321} (2009), 1746--1785.
\bibitem[DCT15]{DCT15} K. De Commer and T. Timmermann, Partial compact quantum groups, \emph{J. Algebra} \textbf{438} (2015), 283--324.
\bibitem[DC17]{DC17} K. De Commer, Actions of compact quantum groups, \emph{Banach Center Publications} \textbf{111} (2017), 33--100.
\bibitem[DCY13]{DCY13} K. De Commer and M. Yamashita, Tannaka-Krein duality for compact quantum homogeneous spaces.
I. General theory, Theory Appl. Categ. \textbf{28} (31) (2013), 1099--1138. 
\bibitem[Dau72]{Dau72} J. Dauns, Simple modules and centralizers, \emph{Trans. Amer. Math. Soc.} \textbf{166} (1972), 457--477.
\bibitem[DVDZ99]{DVDZ99} B. Drabant, A. Van Daele and  Y. H. Zhang, Actions of multiplier hopf algebras,  \emph{Comm. Algebra} \textbf{27} (9) (1999), 4117--4172.
\bibitem[ES80]{ES80} M. Enock and J.-M. Schwartz, Produit crois\'{e} d’une alg\`{e}bre de von Neumann par une alg\`{e}bre de
Kac II, \emph{Publ. RIMS} \textbf{16} (1980), 189--232.
\bibitem[GaSi91]{GaSi91} J.L. Garc\'{i}a and J.J. Sim\'{o}n, Morita equivalences for idempotent rings, \emph{J. Pure Appl. Algebra} \textbf{76} (1991), 39--56. 
\bibitem[GaVi98]{GaVi98} F. Grandjean en E. M. Vitale, Morita equivalence for regular algebras, \emph{Cahiers Topologie G\'{e}om. Diff\'{e}rentielle Cat\'{e}g.} \textbf{39} (2) (1998), 137--153.
\bibitem[Hay96]{Hay96} T. Hayashi, Compact quantum groups of face type, \emph{Publ. Res. Inst. Math. Sci.} \textbf{32} (2) (1996), 351--369.
\bibitem[Kyu86]{Kyu86} S. Kyuno, Equivalence of module categories, \emph{Math. J. Okayama Univ.} \textbf{28} (1986), 147--150.
\bibitem[Mar98]{Mar98} L. Mar\'{i}n, Morita equivalence based on contexts for various categories of modules over associative rings, \emph{J. Pure Appl. Algebra} \textbf{133} (1998), 219--232. 
\bibitem[Pod95]{Pod95} P. Podle\'{s}, Symmetries of quantum spaces. Subgroups and quotient spaces of quantum $SU(2)$ and $SO(3)$ groups, \emph{Comm. Math. Phys.} \textbf{170} (1) (1995), 1--20.
\bibitem[Schn90]{Schn90} H.-J. Schneider, Principal homogeneous spaces for arbitrary Hopf algebras, \emph{Israel J. of Math.} \textbf{72} (1990), 167--195.
\bibitem[Sch02]{Sch02} P. Schauenburg, Morita base change in quantum groupoids, \emph{Locally compact quantum groups and groupoids} (Strasbourg, 2002), \emph{IRMA Lect. Math. Theor. Phys.} \textbf{2},
de Gruyter, Berlin (2003), 79--103.
\bibitem[Sch04]{Sch04} P. Schauenburg, Hopf-Galois and Bi-Galois Extensions, \emph{Galois theory, Hopf
algebras, and semiabelian categories, Fields Inst. Commun.} \textbf{43}, AMS (2004), 469--515.
\bibitem[Ulb87]{Ulb87} K.-H. Ulbrich, Galois extensions as functors of comodules, \emph{Manuscripta Math.} \textbf{59} (1987), 391--397.
\bibitem[Ulb89]{Ulb89} K.-H. Ulbrich, Fiber functors of finite dimensional comodules, \emph{Manuscripta Math.} \textbf{65} (1989),
39--46.
\bibitem[VD94]{VD94} A. Van Daele, Multiplier Hopf algebras, \emph{Trans. Amer. Math. Soc.} \textbf{342} (1994), 917--932.
\bibitem[VD98]{VD98} A. Van Daele, An Algebraic Framework for Group Duality,\emph{Adv. Math} \textbf{140} (1998), 323--366.
\bibitem[VDW15]{VDW15} A. Van Daele and S. Wang, Weak multiplier Hopf algebras I. The main theory, \emph{J. Reine Angew. Math.} \textbf{705} (2015), 155--209.
\bibitem[VDW17]{VDW17} A. Van Daele and S. Wang, Weak multiplier Hopf algebras III. Integrals and duality, \emph{preprint}, arXiv:1701.04951.
\bibitem[VDZ99a]{VDZ99a} A. Van Daele and Y. Zhang, Multiplier Hopf Algebras of Discrete Type, \emph{J. Algebra} \textbf{214} (2) (1999), 400--417.
\bibitem[VDZ99b]{VDZ99b} A. Van Daele and Y. H. Zhang, Galois Theory for Multiplier Hopf Algebras with Integrals, \emph{Algebr. Represent. Theory} \textbf{2} (1999), 83--106.
\bibitem[Ver02]{Ver02} R. Vergnioux, KK-th\'{e}orie \'{e}quivariante et op\'{e}rateur de Julg-Valette pour
les groupes quantiques, thesis, universit\'{e} Paris 7 (2002).
\bibitem[Was89]{Was89} A.Wassermann, Ergodic actions of compact groups on operator algebras. I. General theory,
\emph{Ann. of Math.} \textbf{130} (2) (1989), 273--319.
\end{thebibliography}
\end{document}